\documentclass[12pt]{amsart}
\usepackage{amsfonts, amssymb, latexsym, epsfig} 

\usepackage{amscd,amssymb}
\usepackage{verbatim}
\usepackage{pstricks}
\usepackage{pst-node}
\usepackage[mathscr]{euscript}

\input xy
\xyoption{all}

\newsymbol\pp 1275
\usepackage{tikz}
\usetikzlibrary{matrix,arrows,backgrounds,shapes.misc,shapes.geometric,patterns,calc,positioning}

\usepackage[colorlinks=true, pdfstartview=FitV, linkcolor=blue, citecolor=blue, urlcolor=blue]{hyperref}

\usepackage{fullpage}
\usepackage{graphicx}
\usepackage{enumerate}
\def\VR{\kern-\arraycolsep\strut\vrule &\kern-\arraycolsep}
\def\vr{\kern-\arraycolsep & \kern-\arraycolsep}

\newtheorem{theorem}{Theorem}

\newtheorem{lemma}[theorem]{Lemma}
\newtheorem{prop}[theorem]{Proposition}
\newtheorem{corollary}[theorem]{Corollary}
\newtheorem{conjecture}[theorem]{Conjecture}

\theoremstyle{definition}

\newtheorem{question}[theorem]{Question}

\newtheorem{rmk}{Remark}
\newenvironment{remark}[1][]{\begin{rmk}[#1]\pushQED{\qed}}{\popQED \end{rmk}}

\newtheorem{ex}{Example}
\newenvironment{example}[1][]{\begin{ex}[#1]\pushQED{\qed}}{\popQED \end{ex}}

\newcommand{\Hom}{\operatorname{Hom}}
\newcommand{\End}{\operatorname{End}}
\newcommand{\Ext}{\operatorname{Ext}}

\newcommand{\rep}{\operatorname{rep}}
\newcommand{\Proj}{\operatorname{Proj}}
\newcommand{\SI}{\operatorname{SI}}
\newcommand{\SL}{\operatorname{SL}}
\newcommand{\GL}{\operatorname{GL}}
\newcommand{\PGL}{\operatorname{PGL}}

\newcommand{\ZZ}{\mathbb Z}

\newcommand{\NN}{\mathbb N}

\newcommand{\PP}{\mathbb P}
\newcommand{\coker}{\operatorname{Coker}}
\newcommand{\Ima}{\operatorname{Im}}
\newcommand{\Id}{\operatorname{Id}}

\newcommand{\Mat}{\operatorname{Mat}}

\newcommand{\Stab}{\operatorname{Stab}}

\newcommand{\ddim}{\operatorname{\mathbf{dim}}}
\newcommand{\dd}{\operatorname{\mathbf{d}}}
\newcommand{\ee}{\operatorname{\mathbf{e}}}

\newcommand{\rr}{\mathbf{r}}
\newcommand{\pdim}{\mathsf{pdim}}

\newcommand{\M}{\operatorname{\mathcal{M}}}
\newcommand{\XX}{\operatorname{\mathscr{X}}}

\newcommand{\module}{\operatorname{mod}}
\newcommand{\ind}{\operatorname{ind}}

\newcommand{\trdeg}{\operatorname{tr.deg}}
\newcommand{\rk}{\operatorname{rank}}

\newcommand{\qq}{\operatorname{Quot}}

\newcommand{\key}[1]{\emph{#1}}

\begin{document}
\title{On the invariant theory for acyclic gentle algebras}

\author{Andrew T. Carroll}
\address{University of Missouri-Columbia, Mathematics Department, Columbia, MO, USA}
\email[Andrew Carroll]{carrollat@missouri.edu}

\author{Calin Chindris}
\address{University of Missouri-Columbia, Mathematics Department, Columbia, MO, USA}
\email[Calin Chindris]{chindrisc@missouri.edu}

\date{\today}
\bibliographystyle{plain}

\begin{abstract} In this paper we show that the fields of rational invariants over the irreducible components of the module varieties for an acyclic gentle algebra are purely transcendental extensions. Along the way, we exhibit for such fields of rational invariants a transcendence basis in terms of Schofield\rq{}s determinantal semi-invariants.  

We also show that moduli spaces of modules over regular irreducible component are just a products of projective spaces. 
\end{abstract}

\maketitle
\setcounter{tocdepth}{1}
\tableofcontents

\section{Introduction}
Throughout the article, $k$ always denotes an algebraically closed field of characteristic zero. All algebras (associative and with identity) are assumed to be finite-dimensional over $k$, and all modules are assumed to be finite-dimensional left modules.

One of the fundamental problems in the representation theory of algebras is that of classifying the indecomposable modules. Based on the complexity of these modules, one distinguishes the class of tame algebras and that of wild algebras. According to the remarkable Tame-Wild Dichotomy Theorem of Drozd \cite{Dro}, these two classes of algebras are disjoint and they cover the whole class of algebras. Since the representation theory of  a wild algebra is at least as complicated as that of a free algebra in two variables, and since the later theory is known to be undecidable, one can hope to meaningfully classify the indecomposable modules only for tame algebras. For more precise definitions, see \cite[Chapter XIX]{Sim-Sko-3} and the reference therein.

An interesting task in the representation theory of algebras is to study, for a given finite-dimensional algebra, the geometry of the affine varieties of modules of fixed dimension vectors and the actions of the corresponding products of general linear groups. In particular, it would be interesting to find characterizations of prominent classes of tame algebras via geometric properties of their module varieties. This research direction has attracted much attention during the last two decades (see for example \cite{Bob1}, \cite{Bob-Sko-3}, \cite{BS1}, \cite{Bob-Sko-2}, \cite{Carroll1}, \cite{Domo2} \cite{GeiSch}, \cite{Rie}, \cite{Rie-Zwa-1}, \cite{Rie-Zwa-2}, \cite{SW1}).

In this paper, we seek for characterizations of tame algebras in terms of invariant theory. A first result in this direction was obtained by Skowro{\'n}ski and Weyman in \cite[Theorem 1]{SW1} where they showed that a finite-dimensional algebra of global dimension one is tame if and only if all of its algebras of semi-invariants are complete intersections. Unfortunately, this result does not extend to algebras of higher global dimension (not even of global dimension two) as shown by Kra{\'s}kiewicz in \cite{Kra}. As it was suggested by Weyman, in order to characterize the tameness of an algebra via invariant theory, one should impose geometric conditions on the various moduli spaces of semi-stable modules rather than on the entire algebras of semi-invariants. For more precise details, see Section \ref{remarks-sec}.

A description of the tameness of quasi-tilted algebras in terms of the invariant theory of the algebras in question has been found in \cite{CC9,CC10}. In this paper, we continue this line of inquiry for the class of triangular gentle algebras, which are known to be tame.  Their indecomposable modules can be nicely classified, however these algebras still represent an increase in the level of complexity from the tame quasi-tilted case.  For example, it is possible to construct triangular gentle algebras of arbitrarily large global dimension.  Furthermore, the number of one-parameter families required to parameterize $d$-dimensional indecomposable modules can grow faster than any polynomial in $d$.  

\begin{theorem}\label{mainthm-1} Let $A=kQ/I$ be a triangular gentle algebra, $\dd$ a dimension vector of $A$, and $C \subseteq \module(A,\dd)$ an irreducible component. Then the following statements hold:
\begin{enumerate}
\item the field of rational invariants $k(C)^{\GL(\dd)}$ is a purely transcendental extension of $k$ whose transcendece degree equals the sum of the multiplicities of the indecomposable irreducible regular components occurring in the generic decomposition of $C$;

\item if $C$ is an irreducible regular component then for any weight $\theta \in \ZZ^{Q_0}$ with $C^{ss}_{\theta} \neq \emptyset$ and such that the $\theta$-stable summands of $C$ are regular, the moduli space $\M(C)^{ss}_{\theta}$ is just a product of projective spaces. 
\end{enumerate}
\end{theorem}

Our next main result gives a transcendence basis for $k(C)^{\GL(\dd)}$ when $C$ is an irreducible regular component. We do this via the so called up and down graphs, introduced in the context of triangular gentle algebras by the first author (see \cite{Carroll2}). The up and down graphs are combinatorially defined objects which hold valuable geometric information in that they: $(1)$ give the generic decomposition of irreducible components in module varieties; $(2)$ explicitly describe the generic modules in irreducible components; and $(3)$ allow for explicit computations of the so called generalized Schofield' semi-invariants which in turn are rather remarkable coordinates on moduli spaces of modules for finite-dimensional algebras.

\begin{theorem}\label{mainthm-2} Let $A=kQ/I$ be a triangular gentle algebra, $\module(A,\dd,\rr)$ an irreducible regular component, and 
$$\module(A, \dd, \rr) = \overline{\bigoplus_{i=1}^n\module(A, \dd_i,\rr_i)^{\oplus m_i}}$$
be the generic decomposition of $\module(A,\dd,\rr)$ where $\module(A, \dd_i,\rr_i), 1 \leq i \leq n$, are (pairwise distinct) indecomposable irreducible regular components and $m_1,\ldots, m_n$ are positive integers. 

For each $i$, let $\lambda(i,j), 0\leq j \leq p_i$, be pairwise distinct elements of $k^*$ and $M(\dd_i,\rr_i,\lambda(i,j)), 0\leq j \leq m_i$, be the corresponding generic modules of $\module(A,\dd_i,\rr_i)$. Then, $k(\module(A, \dd,\rr))^{\GL(\dd)}$ is a purely transcendental extension of $k$ of degree $N=\sum m_i$ with transcendence basis
\begin{equation}
   \left\{\frac{\overline{c}^{M(\dd_i,\rr_i,\lambda(i,j))\phantom{+}}}{\overline{c}^{M(\dd_i,\rr_i,\lambda(i,j+1))}};
    i=1,\dotsc, n,\ j=0,\dotsc, m_i-1\right\}.
\end{equation}
\end{theorem}

We point out that the class of tame hereditary algebras is the only other class of tame algebras for which a similar transcendence basis has been constructed (see \cite{R4}). For wild quivers, the corresponding rationality problem, which has been open for more than $45$ years, is one of the most important open problems in rational invariant theory (see \cite{Boh}). It is our hope that the representation-theoretic description of the transcendence basis in the theorem above will inspire the construction of transcendence bases in other situations, too.

In Section \ref{sect:background}, we outline the pertinent notions related to bound quiver algebras.  This includes a description of the module varieties and generic decomposition of irreducible components due to Crawley-Boevey and Schr\"{o}er.  Section \ref{sec:rational-invariants} recalls the basic notions of rational invariants and rational quotients.  It includes a general reduction result for fields of rational invariants over irreducible components in module varieties. In Section \ref{sect:modulispaces}, we first review King's construction of moduli spaces of modules for finite-dimensional modules, and then state another general reduction result that allows one to break a moduli space of modules into smaller ones. We also show that, at least in the tame case, these smaller moduli spaces are rather well-behaved (see Proposition \ref{prop:tamerationality}).  These general results are applied to gentle algebras in Section \ref{sec:gentle-sec}, following a recollection of the construction of their generic modules.  We prove a number of facts about these irreducible components which allow an application of the reduction techniques of Sections \ref{sect:modulispaces} and \ref{sec:rational-invariants}.  In section \ref{proofs-sec} we prove the two main theorems of the paper, and point out some idiosynchracies that arise in the generic decompositions of modules over gentle algebras.  Section \ref{remarks-sec} concludes the article by placing the results in the framework of an effort to characterize tameness of an algebra via invariant theoretic characteristics.  

\subsection*{Acknowledgements}  We would like to thank Harm Derksen and Piotr Dowbor for inspiring discussions on the subject of the paper. The second author was supported by NSF grant DMS-1101383.

\section{Background}\label{sect:background}
\subsection{Bound quiver algebras}\label{sect:thebasics} Let $Q=(Q_0,Q_1,t,h)$ be a finite quiver with vertex set $Q_0$ and arrow set $Q_1$. The two functions $t,h:Q_1 \to Q_0$ assign to each arrow $a \in Q_1$ its tail \emph{ta} and head \emph{ha}, respectively.

A representation $V$ of $Q$ over $k$ is a collection $(V(i),V(a))_{i\in Q_0, a\in Q_1}$ of finite-dimensional $k$-vector spaces $V(i)$, $i \in Q_0$, and $k$-linear maps $V(a) \in \Hom_k(V(ta),V(ha))$, $a \in Q_1$. The dimension vector of a representation $V$ of $Q$ is the function $\ddim V : Q_0 \to \ZZ$ defined by $(\ddim V)(i)=\dim_{k} V(i)$ for $i\in Q_0$. Let $S_i$ be the one-dimensional representation of $Q$ at vertex $i \in Q_0$. By a dimension vector of $Q$, we simply mean a vector $\dd \in \ZZ_{\geq 0}^{Q_0}$.

Given two representations $V$ and $W$ of $Q$, we define a morphism $\varphi:V \rightarrow W$ to be a collection $(\varphi(i))_{i \in Q_0}$ of $k$-linear maps with $\varphi(i) \in \Hom_k(V(i), W(i))$ for each $i \in Q_0$, and such that $\varphi(ha)V(a)=W(a)\varphi(ta)$ for each $a \in Q_1$. We denote by $\Hom_Q(V,W)$ the $k$-vector space of all morphisms from $V$ to $W$. Let $V$ and $W$ be two representations of $Q$. We say that $V$ is a subrepresentation of $W$ if $V(i)$ is a subspace of $W(i)$ for each $i \in Q_0$ and $V(a)$ is the restriction of $W(a)$ to $V(ta)$ for each $a \in Q_1$. In this way, we obtain the abelian category $\rep(Q)$ of all representations of $Q$.

Given a quiver $Q$, its path algebra $kQ$ has a $k$-basis consisting of all paths (including the trivial ones) and the multiplication in $kQ$ is given by concatenation of paths. It is easy to see that any $kQ$-module defines a representation of $Q$, and vice-versa. Furthermore, the category $\module(kQ)$ of $kQ$-modules is equivalent to the category $\rep(Q)$. In what follows, we identify $\module(kQ)$ and $\rep(Q)$, and use the same notation for a module and the corresponding representation.

A two-sided ideal $I$ of $kQ$ is said to be \emph{admissible} if there exists an integer $L\geq 2$ such that $R_Q^L\subseteq I \subseteq R_Q^2$. Here, $R_Q$ denotes the two-sided ideal of $kQ$ generated by all arrows of $Q$. 

If $I$ is an admissible ideal of $kQ$, the pair $(Q,I)$ is called a \emph{bound quiver} and the quotient algebra $kQ/I$ is called the \emph{bound quiver algebra} of $(Q,I)$.  Any admissible ideal is generated by finitely many admissible relations, and any bound quiver algebra is finite-dimensional and basic. Moreover, a bound quiver algebra $kQ/I$ is connected if and only if (the underlying graph of) $Q$ is connected (see for example \cite[Lemma~II.2.5]{AS-SI-SK}).

Up to Morita equivalence, any finite-dimensional algebra $A$ can be viewed as the bound quiver algebra of a bound quiver $(Q_{A},I)$, where $Q_{A}$ is the Gabriel quiver of $A$ (see \cite[Corollary I.6.10 and Theorem~II.3.7]{AS-SI-SK}). (Note that the ideal of relations $I$ is not uniquely determined by $A$.) We say that $A$ is a \emph{triangular} algebra if its Gabriel quiver has no oriented cycles.

Fix a bound quiver $(Q,I)$ and let $A=kQ/I$ be its bound quiver algebra.  A representation $M$ of $A$ (or  $(Q,I)$) is just a representation $M$ of $Q$ such that $M(r)=0$ for all $r \in I$. The category $\module(A)$ of finite-dimensional left $A$-modules is equivalent to the category $\rep(A)$ of representations of $A$. As before, we identify $\module(A)$ and $\rep(A)$, and make no distinction between $A$-modules and representations of $A$.  For each vertex $x \in Q_0$, we denote by $P_x$ the projective indecomposable $A$-module at vertex $x$.  For an $A$-module $M$, we denote by $\pdim M$ its projective dimension.  An $A$-module is called \emph{Schur} if $\End_A(M) \cong k$.  

An algebra $A$ is called \emph{tame} if for each dimension vector $\dd\in \ZZ_{\geq 0}^{Q_0}$, the subcategory $\ind_{\dd}(A)$, whose objects are the indecomposable $\dd$-dimensional $A$-modules, is parametrized, the sense of \cite{Sim-Sko-3}, by a finite family of functors $F_i:=Q_i \otimes_{A_i}-: \ind_1(A_i) \to \module(A)$, $1 \leq i \leq n$, where for each $1 \leq i \leq n $, we have: $(1)$ $A_i=k$ or $A_i=k[t]_{h_i}$ is a localization of $k[t]$; $(2)$ $Q_i$ is an $A-A_i$-bimodule which is finitely generated and free as a right $A_i$-module; and $(3)$ the functor $F_i$ is a representation embedding. For more details, we refer to \cite{DowSko2} and \cite{Sim-Sko-3}. 

For the remainder of this subsection, we assume that $A$ has finite global dimension; this happens, for example, when $Q$ has no oriented cycles. The \emph{Euler form} of $A$ is the bilinear form  $\langle \langle \cdot, \cdot \rangle \rangle_{A} : \ZZ^{Q_0}\times \ZZ^{Q_0} \to \ZZ$ defined by
$$
\langle \langle \dd,\ee \rangle \rangle_{A}=\sum_{l\geq 0}(-1)^l \sum_{i,j\in Q_0}\dim_k \Ext^l_{A}(S_i,S_j)\dd(i)\ee(j).
$$
Note that if $M$ is a $\dd$-dimensional $A$-module and $N$ is an $\ee$-dimensional $A$-module then
$$
\langle \langle \dd,\ee \rangle \rangle_{A}=\sum_{l\geq 0}(-1)^l \dim_k \Ext^l_{A}(M,N).
$$

\subsection{Module varieties and their irreducible components}\label{sect:modvar} Let $\dd$ be a dimension vector of $A=kQ/I$ (or equivalently, of $Q$). The affine variety
$$
\module(A,\dd):=\{M \in \prod_{a \in Q_1} \Mat_{\dd(ha)\times \dd(ta)}(k) \mid M(r)=0, \forall r \in
I \}
$$
is called the \emph{module/representation variety} of $\dd$-dimensional modules/representations of $A$. The affine space $\module(Q,\dd):= \prod_{a \in Q_1} \Mat_{\dd(ha)\times \dd(ta)}(k)$ is acted upon by the base change group $$\GL(\dd):=\prod_{i\in Q_0}\GL(\dd(i),k)$$ by simultaneous conjugation, i.e., for $g=(g(i))_{i\in Q_0}\in \GL(\dd)$ and $V=(V(a))_{a \in Q_1} \in \module(Q,\dd)$,  $g \cdot V$ is defined by $$(g\cdot V)(a)=g(ha)V(a) g(ta)^{-1}, \forall a \in Q_1.$$ 

It can be easily seen that $\module(A,\dd)$ is a $\GL(\dd)$-invariant closed subvariety of $\module(Q,\dd)$, and that the $\GL(\dd)-$orbits in $\module(A,\dd)$ are in one-to-one correspondence with the isomorphism classes of the $\dd$-dimensional $A$-modules. Note that $\module(A, \dd)$ does not have to be irreducible. 

Let $C$ be an irreducible component of $\module(A, \dd)$. We say that $C$ is \emph{indecomposable} if $C$ has a non-empty open subset of indecomposable modules; whenever $\module(A,\dd)$ has such an irreducible component, we say that $\dd$ is a \emph{generic root} of $A$. 

Given a decomposition $\dd=\dd_1+\ldots +\dd_t$ where $\dd_i \in \ZZ^{Q_0}_{\geq 0}, 1 \leq i \leq t$, and $\GL(\dd_i)$-invariant constructible subsets $C_i\subseteq \module(A,\dd_i)$, $1 \leq i \leq t$, we denote by $C_1\oplus \ldots \oplus C_t$ the constructible subset of $\module(A,\dd)$ defined as:  
$$C_1\oplus \ldots \oplus C_t=\{M \in \module(A,\dd) \mid M\simeq \bigoplus_{i=1}^t M_i\text{~with~} M_i \in C_i, \forall 1 \leq i \leq t\}.$$

As shown by de la Pe{\~n}a in \cite[Section 1.3]{delaP} and Crawley-Boevey and Schr{\"o}er in \cite[Theorem~1.1]{C-BS}, any irreducible component of a module variety satisfies a Krull-Schmidt type decomposition. Specifically, if $C$ is an irreducible component of $\module(A,\dd)$ then there are unique generic roots $\dd_1, \ldots, \dd_t$ of $A$ such that $\dd=\dd_1+\ldots +\dd_t$ and
$$
C=\overline{C_1\oplus \ldots \oplus C_t}
$$
for some indecomposable irreducible components $C_i\subseteq \module(A,\dd_i), 1 \leq i \leq t$. Moreover, the indecomposable irreducible components $C_i, 1 \leq i \leq t,$ are uniquely determined by this property. We call $\dd=\dd_1\oplus \ldots \oplus \dd_t$ the \emph{generic decomposition of $\dd$ in $C$}, and $C=\overline{C_1\oplus \ldots \oplus C_t}$ the \emph{generic decomposition of $C$}.

It also follows from \cite[Theorem 1.2]{C-BS} that if $C=\overline{C_1\oplus \ldots \oplus C_t}$ is the generic decomposition of $C\subseteq \module(A,\dd)$ then $\overline{\bigoplus_{i \in J} C_i}$ is an irreducible component of $\module(A,\sum_{i \in J}\dd(i))$ for any subset $J \subseteq \{1,\ldots, t\}$.

For the remainder of this section, we assume that $A$ is a tame algebra and let $\dd$ be a generic root of $A$. Denote by $\ind(A,\dd)$ the constructible subset of $\module(A,\dd)$ consisting of all $\dd$-dimensional indecomposable $A$-modules. 

We know from the work of Dowbor and Skowro{\'n}ski in \cite{DowSko2} that there are finitely many principal open subsets $\mathcal U_i\subseteq \mathbb A^1=k$ and regular morphisms $f_i:\mathcal U_i \to \module(A,\dd)$, $1 \leq i \leq n$, such that:
\begin{itemize}
\item for each $1 \leq i \leq n$, $f_i(\mathcal U_i) \subset \ind(A,\dd)$, and if $f_i(\lambda_1) \simeq f_i(\lambda_2)$ as $A$-modules then $\lambda_1=\lambda_2$;
\item all modules in $\ind(A,\dd)$, except possibly finitely many isoclasses, belong to $\bigcup_{i=1}^n \mathcal F_i$, where each $\mathcal F_i$ is the closure of the image of the  action morphism $\GL(\dd) \times \mathcal U_i \to \module(A,\dd)$ that sends $(g,\lambda)$ to $g\cdot f_i(\lambda)$, i.e. $\mathcal F_i=\overline{\bigcup_{\lambda \in \mathcal U_i} \GL(\dd)f_i(\lambda)}$.
\end{itemize}
(We call $(\mathcal U_i,f_i)$, $1 \leq i \leq n$, \emph{parameterizing pairs} for $\ind(A,\dd)$.)

Consequently, we have that $$\overline{\ind(A,\dd)}=\bigcup_{i=1}^n \mathcal F_i \cup \bigcup_{j=1}^l\overline{\GL(\dd)M_j}$$ where $M_1, \ldots, M_l \in \ind(A,\dd)$. 

Now, let $C \subseteq \module(A,\dd)$ be an indecomposable irreducible component; in particular, $C$ is an irreducible component of $\overline{\ind(A,\dd)}$. From the discussion above, it follows that either:
\begin{itemize}
\item $C=\mathcal F_i$ for some $1 \leq i \leq n$ or;
\item $C=\overline{\GL(\dd)M_j}$ for some $1 \leq j \leq l$.
\end{itemize}

In the tame case, we have the following simple but very useful dimension count:

\begin{lemma}\label{dim-count-irr-comp} Let $A=kQ/I$ be a tame bound quiver algebra, $\dd$ a generic root, and $C \subseteq \module(A,\dd)$ an indecomposable irreducible component. Then,
$$
\dim C-\dim \GL(\dd)+\min\{\dim_k \End_A(M)\mid M \in C\} \in \{0,1\}.
$$
\end{lemma}

\begin{proof} If $C$ is an orbit closure then the expression above is obviously zero. The only other possibility is when $C=\overline{\bigcup_{\lambda \in \mathcal U}\GL(\dd)f(\lambda)}$ with $(\mathcal U,f)$ a parameterizing pair as above. The action morphism $\mu:\GL(\dd)\times \mathcal U \to C, \mu(g,\lambda)=gf(\lambda), \forall (g,\lambda) \in \GL(\dd)\times \mathcal U$, is dominant by construction and, moreover, for a generic $M_0=\mu(g_0,\lambda_0) \in C$:

\begin{align*}
\mu^{-1}(M_0)&=\{(g,\lambda) \in \GL(\dd)\times \mathcal U \mid g\cdot f(\lambda)=g_0\cdot f(\lambda_0)\} \\
&=g_0\Stab_{\GL(\dd)}(f(\lambda_0))\times\{\lambda_0\}.
\end{align*}
In particular, this shows that $\dim \mu^{-1}(M_0)=\dim_k \End_A(M_0)$. Using now the theorem on the generic fiber, we get that $\dim \GL(\dd)+1-\dim C=\dim_k \End_A(M_0)$ which is exactly what we need to prove.
\end{proof}

\section{Rational Invariants}\label{sec:rational-invariants} Let $A=kQ/I$ be a bound quiver algebra, $\dd \in \ZZ_{\geq 0}^{Q_0}$ a dimension vector of $A$, and $C$ a $\GL(\dd)$-invariant irreducible closed subset of $\module(A,\dd)$. 

The field of rational $\GL(\dd)$-invariants on $C$, denoted by $k(C)^{\GL(\dd)}$, is defined as follows:
$$
k(C)^{\GL(\dd)}=\{\varphi \in k(C) \mid g\varphi=\varphi, \forall g \in \GL(\dd) \}.
$$ 
Our motivation for studying such fields of rational invariants is twofold: on one hand, $k(C)^{\GL(\dd)}$ is the function field of the rational quotient (in the sense of Rosenlicht) of $C$ by $\GL(\dd)$ which parametrizes the generic $A$-modules in $C$. Hence, these fields play a key role in the generic representation theory of $A$. 

On the other hand, there are important projective (parameterizing) varieties, such as the Hilbert scheme of points on $\PP^2$ (see \cite{ABCH}) or various moduli spaces of sheaves over rational surfaces (see \cite{S4}), whose function fields can be viewed as fields of rational invariants for bound quiver algebras. Thus, it is desirable to know how methods and ideas from representation theory of algebras can be used to shed light on the rationality question for such varieties.

In what follows, we explain how to reduce the problem of describing fields of rational invariants on irreducible components of module varieties to the case where the irreducible components involved are indecomposable. But first we recall some fundamental facts from rational invariant theory for which we refer the reader to \cite{Rei1} and the reference therein. Let $G$ be a linear algebraic group acting regularly on an irreducible variety $X$. The field $k(X)^G$ of $G$-invariant rational functions on $X$ is always finitely generated over $k$ since it is a subfield of $k(X)$ which is finitely generated over $k$. A \emph{rational quotient} of $X$ by (the action of) $G$ is an irreducible variety $Y$ such that $k(Y)=k(X)^G$ together with the dominant rational map $\pi:X \dashrightarrow Y$ induced by the inclusion $k(X)^G \subset k(X)$. In \cite{Ros}, Rosenlicht shows that there is always a rational quotient of $X$ by $G$, which is uniquely defined up to birational isomorphism. In fact, it is proved that there is a $G$-invariant open and dense subset $X_0$ of $X$ such that the restriction of $\pi$ to $X_0$ is a dominant regular  morphism and $\pi^{-1}(\pi(x))=Gx$ for all $x \in X_0$ (see \cite[Theorem 2]{Ros} or \cite[Section 1.6]{Bri}). Furthermore, one can show that a rational quotient $\pi:X \dashrightarrow Y$ satisfies the following universal property (see \cite[Remark 2.4 and 2.5]{Rei1} or  \cite[Section 2.4]{PopVin}): Let $\rho:X\dashrightarrow Y'$ be a rational map such that $\rho^{-1}(\rho(x))=Gx$ for $x \in X$ in general position. Then there exists a rational map
$$
\overline{\rho}:Y \dashrightarrow Y'
$$
such that $\rho=\overline{\rho} \circ \pi$. If in addition $\rho$ is dominant then $\overline{\rho}$ becomes a birational isomorphism. One usually writes $X/G$ in place of $Y$ and call it \emph{the} rational quotient of $X$ by $G$.

\begin{remark} From the discussion above and using the theorem on the generic fiber, one can immediately see that
$$
\trdeg k(X)^{G}=\dim X-\dim G +\min \{\dim \Stab_G(x) \mid x \in X\}.
$$
This formula combined with Lemma \ref{dim-count-irr-comp} shows that for a tame algeba $A$, a dimension vector $\dd$, and an indecomposable irreducible component $C \subseteq \module(A,\dd)$, $\trdeg k(C)^{\GL(\dd)} \in \{0,1\}$.
\end{remark}

In what follows, if $R$ is an integral domain, we denote its field of fractions by $\qq(R)$. Moreover, if $K/k$ is a field extension and $m$ is a positive integer, we define $S^m(K/k)$ to be the field $(\qq(K^{\otimes m}))^{S_m}$ which is, in fact, the same as $\qq((K^{\otimes m})^{S_m})$ since $S_m$ is a finite group.

Now we are ready to prove the following reduction result (compare to \cite[Proposition 4.7]{CC9}): 

\begin{prop}\label{reduction-rational-inv-prop} Let $C \subseteq \module(A,\dd)$ be an irreducible component and let 
$$
C=\overline{\bigoplus_{i=1}^n C_i^{\oplus m_i}\oplus \bigoplus_{j=n+1}^l C_j}
$$ be the generic decomposition of $C$ where $C_i \subseteq \rep(\Lambda,\dd_i)$, $1 \leq i \leq l$, are indecomposable irreducible components with $C_{n+1}, \dots, C_l$ orbit closures, $m_1,\ldots, m_n$ are positive integers, and $\dd_i\neq \dd_j, \forall 1 \leq i\neq j\leq n$. Then,
$$
k(C)^{\GL(\dd)} \simeq \qq(\bigotimes_{i=1}^n S^{m_i}(k(C_i)^{\GL(\dd_i)}/k)).
$$
\end{prop}

\begin{proof} Note that we can write 
$$
C=\overline{C' \oplus C''}
$$
where:
\begin{itemize}
\item $C'=\overline{\bigoplus_{i=1}^n C_i^{\oplus m_i}}$ is an irreducible component of $\module(A,\dd')$ with $\dd'=\sum_{i=1}^n \dd(i)$;\\

\item $C''=\overline{\GL(\dd'')M}$ where $\dd''=\sum_{j=n+1}^l \dd(j)$.
\end{itemize}

Now, Proposition 4.7 in \cite{CC9} tells us that 
$$
k(C\rq{})^{\GL(\dd\rq{})} \simeq \qq(\bigotimes_{i=1}^n S^{m_i}(k(C_i)^{\GL(\dd_i)}/k)).
$$

The proof will follow from the lemma below.
\end{proof}

\begin{lemma}\label{lemma-orbit-rational-inv} Assume that $C \subseteq \module(A,\dd)$ can be decomposed as $$C=\overline{C'\oplus C''}$$ where $C' \subseteq \module(A,\dd')$ is a $\GL(\dd')$-invariant irreducible closed subvariety, $C''=\overline{\GL(\dd'')M} \subseteq \module(A,\dd'')$, and $\dd=\dd'+\dd''$. Then,
$$
k(C)^{\GL(\dd)} \simeq k(C')^{\GL(\dd')}.
$$  
\end{lemma}

\begin{proof} Fix a decomposition $k^{\dd(v)}=k^{\dd'(v)}\oplus k^{\dd''(v)}$ for all $v \in Q_0$, and then embed $C'\times C''$ diagonally into $C$, and $\GL(\dd')\times\GL(\dd'')$ diagonally into $\GL(\dd)$. Denote $\GL(\dd)$ by $G$, $\GL(\dd')\times \{\textbf{1}\}$ by $H$, and let $S=C'\times \{M\}$. Then, it is easy to see that $S$ is an irreducible $H$-invariant closed subvariety of $C$ such that:

\begin{enumerate}
\item $\overline{GS}=C$;
\item for a generic $s \in S$, we have $Gs \cap S=Hs$. 
\end{enumerate}

Now, let $\pi:C \dashrightarrow C/G$ be the rational quotient of $C$ by $G$. Then, (1) ensures that the restriction $\rho$ of $\pi$ to $S$  is a well-defined dominant rational map, and (2) simply says that for generic $s \in S$, $\rho^{-1}(\rho(s))=Hs$. It follows from the universal property for rational quotients that $\rho$ is the rational quotient of $S$ by $H$, and so $k(C)^{\GL(\dd)}\simeq k(S)^H=k(C')^{\GL(\dd')}$.
\end{proof}

\begin{remark} This lemma tells us that, for the purposes of computing rational invariants, we can always get rid of the orbit closures that occur in a generic decomposition. If, additionally, the other irreducible components that occur in a generic decomposition can be separated as in Proposition \ref{reduction-rational-inv-prop} then we have a further reduction in the fields of rational invariants.
\end{remark}

\section{Moduli spaces of modules}\label{sect:modulispaces} 
Let $A=kQ/I$ be a bound quiver algebra and let $\dd \in \ZZ^{Q_0}_{\geq 0}$ be a dimension vector of $A$. The ring of invariants $\text{I}(A,\dd):= k[\module(A,\dd)]^{\GL(\dd)}$ turns out to be precisely the base field $k$ since $A$ is finite-dimensional. However, the action of the subgroup $\SL(\dd) \subseteq \GL(\dd)$, defined by
$$
\SL(\dd)=\prod_{i \in Q_0}\SL(\dd(i),k),
$$
provides us with a highly non-trivial ring of semi-invariants. Note that any $\theta \in \ZZ^{Q_0}$ defines a rational character $\chi_{\theta}:\GL(\dd) \to k^*$ by $$\chi_{\theta}((g(i))_{i \in Q_0})=\prod_{i \in Q_0}(\det g(i))^{\theta(i)}.$$ In this way, we can identify $\Gamma=\ZZ ^{Q_0}$ with the group $X^\star(\GL(\dd))$ of rational characters of $\GL(\dd)$, assuming that $\dd$ is a
sincere dimension vector. In general, we have only the natural epimorphism $\Gamma \to X^*(\GL(\dd))$. We also refer to the rational characters of $\GL(\dd)$ as (integral) weights of $A$ (or $Q$).

Let us now consider the ring of semi-invariants $\SI(A,\dd):= k[\module(A,\dd)]^{\SL(\dd)}$. As $\SL(\dd)$ is the commutator subgroup of $\GL(\dd)$ and $\GL(\dd)$ is linearly reductive, we have $$\SI(A,\dd)=\bigoplus_{\theta \in X^\star(\GL(\dd))}\SI(A,\dd)_{\theta},
$$
where $$\SI(A,\dd)_{\theta}=\lbrace f \in k[\module(A,\dd)] \mid g f= \theta(g)f \text{~for all~}g \in \GL(\dd)\rbrace$$ is
called \key{the space of semi-invariants} on $\module(A,\dd)$ \key{of weight $\theta$}. For an irreducible component $C \subseteq \module(A,\dd)$, we similarly define the ring of semi-invariants $\SI(C):=k[C]^{\SL(\dd)}$, and the space $\SI(C)_{\theta}$ of semi-invariants on $C$ of weight $\theta \in \ZZ^{Q_0}$.

Following King \cite{K}, an $A$-module $M$ is said to be \emph{$\theta$-semi-stable} if $\theta(\ddim M)=0$ and $\theta(\ddim M')\leq 0$ for all submodules $M' \leq M$. We say that $M$ is \emph{$\theta$-stable} if $M$ is non-zero, $\theta(\ddim M)=0$, and $\theta(\ddim M')<0$ for all submodules $\{0\} \neq M' < M$. We denote by $\module(A)^{ss}_{\theta}$ the full subcategory of $\module(A)$ consisting of the $\theta$-semi-stable modules. It is easy to see that $\module(A)^{ss}_{\theta}$ is a full exact abelian subcategory of $\module(A)$ which is closed under extensions and whose simple objects are precisely the $\theta$-stable modules. Moreover, $\module(A)^{ss}_{\theta}$ is Artinian and Noetherian, and hence every $\theta$-semi-stable $A$-module has a Jordan-H{\"o}lder filtration in $\module(A)^{ss}_{\theta}$.

Now, let us consider the (possibly empty) open subsets
$$\module(A,\dd)^{ss}_{\theta}=\{M \in \module(A,\dd)\mid M \text{~is~}
\text{$\theta$-semi-stable}\}$$
and $$\module(A,\dd)^s_{\theta}=\{M \in \module(A,\dd)\mid M \text{~is~}
\text{$\theta$-stable}\}$$
of $\dd$-dimensional $\theta$(-semi)-stable $A$-modules. Using methods from Geometric Invariant Theory, King showed in \cite{K} that the projective variety
$$
\M(A,\dd)^{ss}_{\theta}:=\Proj(\bigoplus_{n \geq 0}\SI(A,\dd)_{n\theta})
$$
is a GIT-quotient of $\module(A,\dd)^{ss}_{\theta}$ by the action of $\PGL(\dd)$ where $\PGL(\dd)=\GL(\dd)/T_1$ and $T_1=\{(\lambda \Id_{\dd(i)})_{i \in Q_0} \mid \lambda \in k^*\} \leq \GL(\dd)$. Moreover, there is a (possibly empty) open subset $\M(A,\dd)^s_{\theta}$ of $\M(A,\dd)^{ss}_{\theta}$ which is a geometric quotient of $\module(A,\dd)^s_{\theta}$ by $\PGL(\dd)$. We say that $\dd$ is a \emph{$\theta$-(semi-)stable dimension vector} if $\module(A,\dd)^{(s)s}_{\theta} \neq \emptyset$. 


For an irreducible component $C \subseteq \module(A,\dd)$, we similarly define $C^{ss}_{\theta}, C^s_{\theta}$, $\M(C)^{ss}_{\theta}$, and $\M(C)^s_{\theta}$. One then has that the points of 
$\M(C)^s_{\theta}$ correspond bijectively to the isomorphism classes of $\theta$-stable modules in $C$. We say that $C$ is \emph{$\theta$-(semi-)stable} if $C^{(s)s}_{\theta} \neq \emptyset$.

\begin{remark} \label{rmk-fine-moduli-rational-inv} Let $C \subseteq \module(A,\dd)$ be an irreducible component and let $\theta \in \ZZ^{Q_0}$ be so that $C^s_{\theta} \neq \emptyset$. Let $\pi:C \dashrightarrow \M(C)^s_{\theta}$ be the dominant rational morphism which is represented by the geometric quotient morphism $C^s_{\theta} \to \M(C)^s_{\theta}$. Since for any $M \in C^s_{\theta}$, $\pi^{-1}(\pi(M))=\GL(\dd)M$, we conclude that $\M(C)^s_{\theta}$ is the rational quotient of $C$ by $\GL(\dd)$. In particular, we have that $k(C)^{\GL(\dd)}\simeq k(\M(C)^{ss}_{\theta})$.
\end{remark}

A $\theta$-semi-stable irreducible component $C\subseteq \module(A,\dd)$ is said to be \emph{$\theta$-well-behaved} if, whenever $\dd'$ is the dimension vector of a factor of a Jordan-H{\"o}lder filtration in $\module(A)^{ss}_{\theta}$ of a generic $A$-module in $C$ and $C_1,C_2 \subseteq \module(A,\dd')$ are two distinct irreducible components, $C_{1,\theta}^s \cap C_{2,\theta}^s = \emptyset$. (We should point out that this notion is slightly more general than the one in \cite[Section 3.3]{CC10}.)

\begin{remark} It follows from the work of Bobinski and Skowro{\'n}ski in \cite{BS1} that for a tame quasi-tilted algebra, any $\theta$-semi-stable irreducible component is $\theta$-well-behaved.

We will prove in Section \ref{sec:gentle-sec} that the same is true for triangular gentle algebras.
\end{remark}

Let $C$ be a $\theta$-well-behaved irreducible component of $\module(A,\dd)$. We say that
$$
C=C_1\pp \ldots \pp C_l
$$
is \emph{the $\theta$-stable decomposition of $C$} if:
\begin{itemize}
\item the $C_i\subseteq \module(A,\dd_i)$, $1 \leq i \leq l$, are $\theta$-stable irreducible components;
\item the generic $A$-module $M$ in $C$ has a finite filtration $0=M_0\subseteq M_1 \subseteq \dots \subseteq M_l=M$ of submodules such that each factor $M_j/M_{j-1}$, $1 \leq j \leq l$, is isomorphic to a $\theta$-stable module in one the $C_1,\ldots, C_l$, and the sequence $(\ddim M_1/M_0, \ldots, \ddim M/M_{l-1})$ is the same as $(\dd_1,\ldots,\dd_l)$ up to permutation.
\end{itemize}
We call $C_1, \ldots, C_l$ the $\theta$-stable summands of $C$. To prove the existence and uniqueness of the $\theta$-stable decomposition of $C$, first note that the irreducible variety $C^{ss}_{\theta}$ is a disjoint union of sets of the form $\mathcal{F}_{(C_i)_{1 \leq i \leq l}}$, where each $\mathcal{F}_{(C_i)_{1 \leq i \leq l}}$ consists of those modules $M \in C$ that have a finite filtration $0=M_0\subseteq M_1 \subseteq \dots \subseteq M_l=M$ of submodules with each factor $M_j/M_{j-1}$ isomorphic to a $\theta$-stable module in one the $C_i$, $1 \leq i \leq l$. (Note that the $\theta$-well-behavedness of $C$ is used to ensure that the union above is indeed disjoint.) Next, it is not difficult to show that each $\mathcal{F}_{(C_i)_{1 \leq i \leq l}}$ is constructible (see for example \cite[Section 3]{C-BS}). Hence, there is a unique (up to permutation) sequence $(C_i)_{1 \leq i \leq l}$ of $\theta$-stable irreducible components for which $\mathcal{F}_{(C_i)_{1 \leq i \leq l}}$ contains an open and dense subset of $C^{ss}_{\theta}$ (or $C$).

Now, we are ready to state the following reduction result from \cite[Theorem 1.4]{CC10}):

\begin{theorem}\label{theta-stable-decomp-thm} Let $A=kQ/I$ be a bound quiver algebra and let $C \subseteq \module(A,\dd)$ be a  $\theta$-well-behaved irreducible component where $\theta$ is an integral weight of $A$. Let $$C=m_1\cdot C_1 \pp \ldots \pp m_n\cdot C_n $$ be the $\theta$-stable decomposition of $C$ where $C_i \subseteq \module(A,\dd_i)$, $1 \leq i \leq n$, are $\theta$-stable irreducible components and $\dd_i\neq \dd_j$ for all $1 \leq i\neq j \leq n$. Assume that $C$ is a normal variety and $\overline{\bigoplus_{i=1}^n C_i^{\oplus m_i}} \subseteq C$. 
Then,
$$
\mathcal{M}(C)^{ss}_{\theta} \cong S^{m_1}(\mathcal{M}(C_1)^{ss}_{\theta}) \times \ldots \times S^{m_n}(\mathcal{M}(C_n)^{ss}_{\theta}).
$$
\end{theorem}

Note that this reduction result allows us to ``break'' a moduli space of modules into smaller ones. We show next that these smaller moduli spaces are rather well behaved, at least in the tame case.

\begin{prop}\label{prop:tamerationality} Let $A=kQ/I$ be a tame bound quiver algebra, $\dd$ a generic root of $A$, and $C \subseteq \module(A,\dd)$ an indecomposable irreducible component. Then, the following statements hold.
\begin{enumerate}
\item For any weight $\theta \in \ZZ^{Q_0}$ with $C^{ss}_{\theta} \neq \emptyset$, $\M(C)^{ss}_{\theta}$ is either a point or projective curve.
\item If $\theta \in \ZZ^{Q_0}$ is so that $C^s_{\theta} \neq \emptyset$ then $\M(C)^{ss}_{\theta}$ is rational. If, in addition, $C$ is normal then $\M(C)^{ss}_{\theta}$ is either a point or $\mathbb P^1$.
\item If $C$ is a Schur component then $k(C)^{\GL(\dd)}$ is a rational field of transcendence degree at most one. 
\end{enumerate}
\end{prop}

\begin{proof} If $C$ is an obit closure then $(1)-(3)$ are obviously true. Otherwise, we have seen in Section \ref{sect:modvar} that
$$C=\overline{\bigcup_{\lambda \in \mathcal U}\GL(\dd)f(\lambda)}$$
where $(\mathcal U\subseteq k^*,f:\mathcal U \to C)$ is a parameterizing pair.

$(1)$ Let $\theta \in \ZZ^{Q_0}$ be an integral weight and 
$
\pi:C^{ss}_{\theta}\to \M(C)^{ss}_{\theta}
$
be the quotient morphism. Then, for a generic module $M_0 \in C^{ss}_{\theta}$, we have:
$$
\dim C-\dim \M(C)^{ss}_{\theta}=\dim \pi^{-1}(\pi(M_0))\geq \dim \overline{\GL(\dd)M_0}.
$$
So, $\dim \M(C)^{ss}_{\theta} \leq \dim C-\dim \GL(\dd)+\min \{\dim_k \End_A(M)\mid M \in C\}$, which is at most $1$ by Lemma \ref{dim-count-irr-comp}. This proves $(1)$.

$(2)$ Choose an non-empty open subset $X_0 \subseteq C$ such that $X_0 \subseteq \bigcup_{\lambda \in \mathcal U} \GL(\dd)f(\lambda) \cap C^s_{\theta}$. We can certainly assume that $X_0$ is $\GL(\dd)$-invariant since otherwise we can simply work with $\bigcup_{g \in \GL(\dd)}gX_0$. It is then easy to see that $\mathcal U_0:=\{\lambda \in \mathcal U \mid f(\lambda) \in X_0\}$ is a non-empty open subset of $\mathcal U$. 

Now, let $\varphi:\mathcal U_0 \to \M(C)^{ss}_{\theta}$ be the morphism defined by $\varphi(\lambda)=\pi(f(\lambda)), \forall \lambda \in \mathcal U_0$. Notice that $\varphi$ is injective. Indeed, let $\lambda_1, \lambda_2 \in \mathcal U_0$ be so that $\varphi(\lambda_1)=\varphi(\lambda_2)$. Then, $\pi(f(\lambda_1))=\pi(f(\lambda_2))$ which is equivalent to $\overline{GL(\dd)f(\lambda_1)}\cap \overline{\GL(\dd)f(\lambda_2)}\cap C^{ss}_{\theta}\neq \emptyset$. Since $f(\lambda_1)$ and $f(\lambda_2)$ are both $\theta$-stable, we know that their orbits are closed in $C^{ss}_{\theta}$. It is now clear that $f(\lambda_1)\simeq f(\lambda_2)$ which implies that $\lambda_1=\lambda_2$. 

The injectivity of $\varphi$ together with the fact that $\dim \mathcal U_0=1$ and $\dim \M(C)^{ss}_{\theta} \leq 1$ implies that $\varphi$ is an injective dominant morphism, and hence is birational. This shows that $\M(C)^{ss}_{\theta}$ is a rational projective curve.

 If, in addition, $C$ is normal then so is any good quotient of $C$. In particular, under this extra assumption, $\M(C)^{ss}_{\theta}$ is a rational normal projective curve, i.e. $\M(C)^{ss}_{\theta}$ is precisely $\mathbb P^1$. 

(3) Notice that $\{\lambda \in \mathcal U \mid \dim_k \End_A(f(\lambda))=\min \{\dim_k \End_A(M)\mid M \in C\}\}$ is a non-empty open subset of $\mathcal U$. From this we deduce that $C$ contains infinitely many, pairwise non-isomorphic, Schur $A$-modules. Hence, $C$ must contain a homogenous Schur $A$-module $M$ by \cite[Theorem D]{CB5}. It has been proved in \cite[Lemma 11]{CC12} that $M$ is $\theta^M$-stable where $\theta^M$ is a specific integral weight associated to $M$ (see also Section \ref{proofs-sec}). Using Remark \ref{rmk-fine-moduli-rational-inv} and part $(2)$, we can now see that
$$
k(C)^{\GL(\dd)}\simeq k(\M(C)^{ss}_{\theta^M})\simeq k(t).
$$
\end{proof}

\begin{remark} As a consequence of the proposition above and Theorem \ref{theta-stable-decomp-thm},  we obtain that for a tame algebra $A$ and an irreducible component $C \subseteq \module(A,\dd)$ for which Theorem \ref{theta-stable-decomp-thm} is applicable, the moduli space $\M(C)^{ss}_{\theta}$ is a rational variety.
\end{remark}

\section{Acyclic gentle algebras}\label{sec:gentle-sec}

Gentle algebras are a particularly well-behaved class of string algebras that have recently enjoyed a resurgence in popularity, primarily due to their appearance in the study of cluster algebras arising from unpunctured surfaces (see \cite{ABCP}). String algebras are (non-hereditary) tame algebras whose indecomposable modules can be parameterized by certain walks on the underlying quiver $Q$ not passing through the ideal $I$, and whose irreducible morphisms can be described by operations on these walks \cite{BR}.  The invariant theory for a particular class of triangular gentle algebras was first studied in detail by Kra\'{s}kiewicz and Weyman \cite{Kraskiewicz:2011aa}, and then by the second author in \cite{Carroll1, Carroll2, Carroll:Thesis}. In these latter works, the irreducible components of triangular gentle algebras are determined. Their rings of semi-invariants are shown to be semigroup rings, and the generic modules in irreducible components are constructed. This allowed for the calculation of the quotients $\M(C)^{ss}_\theta$ for some special choices of $C$ and $\theta$.  

A bound quiver algebra $kQ/I$ is called a \emph{gentle algebra} if the following properties hold:
\begin{enumerate}
\item for each vertex $i\in Q_0$ there are at most two arrows with
  head $i$, and at most two arrows with tail $i$;
\item for any arrow $b\in Q_1$, there is at most one arrow $a\in Q_1$ and at most one arrow $c\in Q_1$ such that $ab\notin I$ and $bc\notin I$;
\item for each arrow $b\in Q_1$ there is at most one arrow $a\in Q_1$ with $ta=hb$ (resp. at most one arrow $c\in Q_1$ with $hc=tb$) such that $ab\in I$ (resp. $bc\in I$);
\item $I$ is generated by paths of length 2.
\end{enumerate}
The bound quiver algebra $kQ/I$ is called a \key{gentle algebra}.  

In \cite{Carroll1}, colorings of a quiver were introduced to understand the module varieties of triangular gentle algebras. A \key{coloring} $c$ of a quiver $Q$ is a surjective map $c: Q_1 \rightarrow S$, where $S$ is some finite set whose elements we call colors, such that $c^{-1}(s)$ is a directed path for all $s\in S$. Given a coloring $c$ of $Q$, define the coloring ideal $I_c$ to be the two-sided ideal in $kQ$ generated by monochromatic paths of length two, i.e. $I_c=\langle ba \mid c(b)=c(a) \textrm{ and } ha=tb\rangle$.  

\begin{prop}{\cite[Proposition 2.6]{Carroll1}}
If $k Q/I$ is a triangular gentle algebra, then there is a coloring $c$ of $Q$ such that $I=I_c$.
\end{prop}

Let $c$ be a coloring of $Q$, $A=kQ/I_c$, and $\dd$ a dimension vector of $A$. It turns out that one can exhibit all irreducible components of $\module(A, \dd)$ by introducing the concept of a rank function.  A \key{rank function} is a map $\rr:Q_1\rightarrow \NN$ satisfying the inequalities $$\rr(a)+\rr(b)\leq \dd(i)$$ whenever $ha=tb=i$ and $c(a)=c(b)$ (with the degenerate inequalities $\rr(a)\leq \dd(i)$ or $\rr(b)\leq \dd(i))$ when either $b$ or $a$ fails to exist).  The rank function $\rr$ is called \key{maximal} if it is so under the coordinate-wise partial order (namely $\rr'\leq \rr$ if and only if $\rr'(a)\leq \rr(a)$ for all $a\in Q_1$). Note that a rank function $\rr$ defines a closed subvariety of $\module(A,\dd)$:
$$\module(A, \dd, \rr):=\{V\in \module(A,\dd)\mid \rk V(a)\leq \rr(a), \forall a\in Q_1\}.$$  
It was shown in \cite{Carroll1} that every irreducible component of
$\module(A, \dd)$ is of the form $\module(A, \dd, \rr)$ for $\rr$
maximal.  These varieties can then be viewed as products of varieties of complexes taken along each colored path.  DeConcini-Strickland \cite{DeConcini-Strickland} showed that varieties of complexes are normal, and thus $\module(A, \dd, \rr)$ is normal for any choice of $\dd,\rr$. 

\subsection{Up and down graphs}
We now focus exclusively on the case in which $A=kQ/I_c$ is a
triangular gentle algebra.  Under this restriction, each vertex has at
most two colors incident to it.  Denote by $\XX$ the set of pairs $(i,
s)\in Q_0\times S$ such that $s$ is a color incident to $i$.  A \key{sign function} is a map $\epsilon: \XX \rightarrow \{\pm 1\}$ satisfying the property that $\epsilon(i, s) = -\epsilon(i, s')$ if $(i, s), (i, s')$ are distinct elements of $\XX$.  

The \key{up and down graph} $\Gamma(Q, c, \dd, \rr, \epsilon)$ is the
directed graph with vertices $\{v_j^{i}\mid i\in Q_0, j=1,\dotsc, \dd(i)\}$ and arrows $\{f_j^{a}\mid a\in Q_1,\ j=1,\dotsc, \rr(a)\}$ so that 
\begin{align*}
tf_j^{a} &= \begin{cases} v_{j}^{ta} & \textrm{ if } \epsilon(ta, c(a))=\phantom{-}1\\ v_{\dd(ta)-j+1}^{ta} & \textrm{ if } \epsilon(ta, c(a))=-1\end{cases}\\
hf_j^{a} &= \begin{cases} v_{\dd(ha)-j+1}^{ha} & \textrm{ if } \epsilon(ha, c(a))=\phantom{-}1\\ v_j^{ha} & \textrm{ if } \epsilon(ha, c(a))=-1\end{cases}
\end{align*}

There is an obvious morphism of quivers $\pi: \Gamma(Q, c, \dd, \rr, \epsilon) \rightarrow Q$ with $\pi(v_j^i) = i$ and $\pi(f_j^a) = a$.  This, in turn, gives rise to a pushforward map $\pi_*: \rep(\Gamma(Q, c, \dd, \rr, \epsilon)) \rightarrow \rep(Q)$ where if $V=(V(v_j^i), V(f_j^a))\in \rep(\Gamma(Q, c, \dd, \rr, \epsilon))$, then 
\begin{align*}
\pi_*(V)(i) &:= \bigoplus\limits_{j=1,\dotsc, \dd(i)} V(v_j^i)\\
\pi_*(V)(a) &:=\sum\limits_{j=1,\dotsc, \rr(a)} V(f_j^a)
\end{align*}
Notice that since $\rr$ is a rank function, by definition of $\Gamma(Q, c,\dd,\rr,\epsilon)$, $\pi_*(V)(b)\circ
\pi_*(V)(a)=0$ whenever $ha=tb$ and $c(a)=c(b)$.  Therefore the image
of $\pi_*$ lies in $\module(A)$.  Furthermore, each vertex in $\Gamma$ is incident
to at most two arrows, and so the connected components of  $\Gamma(Q,
c, \dd, \rr,\epsilon)$ are either chains or cycles (components which
we refer to as \key{strings} and \key{bands}, respectively).  Let $B$ be the set of band components in $\Gamma$, and for $b\in B$ choose a vertex $\Theta(b)$ contained in $b$.  For each
$\underline{\lambda} =(\lambda_b)_{b\in B} \in k^B$, the representation $\widetilde{M}(Q, c, \dd,
\rr,\epsilon, \underline{\lambda})$ of $\Gamma(Q, c, \dd, \rr, \epsilon)$
is defined to be the representation with $\widetilde{M}(v_j^i)=k$ for all $v_j^i\in \Gamma(Q, c,
\dd, \rr, \epsilon)_0$ with the action of the arrows
given by
\begin{align*}
\widetilde{M}(f_j^a)= \begin{cases} \lambda_b & \textrm{ if } hf_j^a=\Theta(b) \textrm{ and } \epsilon(ha, c(a))=-1\\ 
1 & \textrm{ otherwise}\end{cases} 
\end{align*}
The \key{up and down module} $M(Q, c, \dd, \rr, \epsilon,
\underline{\lambda})$ is defined to be $\pi_*(\widetilde{M}(Q, c, \dd,
\rr, \epsilon, \underline{\lambda}))$.  When the quiver $Q$ and coloring
$c$ are understood from context, we will simply write
$\Gamma(\dd,\rr,\epsilon)$ and $M(\dd, \rr,\epsilon, \underline{\lambda})$.

\begin{example}
  \label{ex:up-down-graphs}
Consider the following bound quiver $(Q, I_c)$ with the coloring as indicated by the type of arrow.
  \begin{align*}
    \xymatrix@C=3cm@R=2cm{
    1 \ar[r]^{r_1} \ar@{..>}[dr]_<<<<<<{g_1} & 2 \ar[r]^{r_2}
  \ar@{-->}[dr]_<<<<<<<{p_2} & 3 \\ 4 \ar@{-->}[ur]^<<<<<<<{p_1}
  \ar@{~>}[r]_{b_1} & 5 \ar@{..>}[ur]^<<<<<<<{g_2} \ar@{~>}[r]_{b_2} & 6}
\end{align*}

We consider the dimension vector $\dd$ and rank function $\rr$ indicated in the diagram below, with dimensions in boxes and ranks as decorations of the arrows.
\begin{align*}
\xymatrix@C=3cm@R=2cm{
    *+[F]{3} \ar[r]^{3} \ar@{..>}[dr]_<<<<<<{2} & *+[F]{4} \ar[r]^{1}
  \ar@{-->}[dr]_<<<<<<<{2} & *+[F]{1} \\ *+[F]{2} \ar@{-->}[ur]^<<<<<<<{2}
  \ar@{~>}[r]_{2} & *+[F]{3} \ar@{..>}[ur]^<<<<<<<{1} \ar@{~>}[r]_{1} & *+[F]{2}}
  \end{align*}
  
Take $\epsilon$ as indicated, where we place $\epsilon(x, c(a))$ on the arrow $a$ near the vertex $x$:

\begin{align*}
    \xymatrix@C=4cm@R=2cm{
    1 \ar[r]^<<<<<<{-}^>>>>>>>{-} \ar@{..>}[dr]_<<<<<<{+}_>>>>>>>>>{-} & 2 \ar[r]^<<<<<{-}^>>>>>>{+}
  \ar@{-->}[dr]_<<<<<<<{+}_>>>>>>>{+} & 3 \\ 4 \ar@{-->}[ur]^<<<<<<<{-}^>>>>>>>{+}
  \ar@{~>}[r]_<<<<<<{+}_>>>>>>{+} & 5 \ar@{..>}[ur]^<<<<<<<{-}^>>>>>>{-} \ar@{~>}[r]_<<<<<<{+}_>>>>>>{-} & 6}
\end{align*}
  
The up and down graph $\Gamma(\dd, \rr, \epsilon)$ consists of one band component and one string component.  One possible choice of $\Theta(b)$ for the unique band is the vertex framed by a circle here:
  \begin{equation*}
\xymatrix@R=.05in@C=.95in{
v_1^{(1)} \ar@[|(2)]@{-}@{->}[ddrr] \ar@[|(2)]@{-}@{..>}[dddddrr]&& {v_1^{(2)}} \ar@[|(9)]@{-}@{-->}[ddddddrr] && v_1^{(3)} \\
v_2^{(1)} \ar@[|(2)]@{-}@{->}[rr] \ar@[|(2)]@{-}@{..>}[dddddrr] && v_2^{(2)} \ar@[|(2)]@{-}@{-->}[ddddrr] && \\
{v_3^{(1)}} \ar@[|(9)]@{-}@{->}[uurr] && v_3^{(2)}  && \\
		&&		v_4^{(2)} \ar@[|(2)]@{-}@{->}[uuurr] && \\
&&&&\\
v_1^{(4)}\ar@[|(2)]@{-}@{~>}[ddrr] \ar@[|(2)]@{-}@{-->}[uuurr] && v_1^{(5)} \ar@[|(2)]@{-}@{~>}[rr]&& *+[o][F]{v_1^{(6)}} \\
v_2^{(4)} \ar@[|(2)]@{-}@{-->}[uuurr] \ar@[|(2)]@{-}@{~>}[rr] && v_2^{(5)}  && {v_2^{(6)}}\\
		&&		v_3^{(5)} \ar@{..>}@[|(2)][uuuuuuurr]		&& }
\end{equation*}

Then $M(\dd,\rr,\epsilon,\lambda)$ is the module given by 
\begin{align*}
\includegraphics{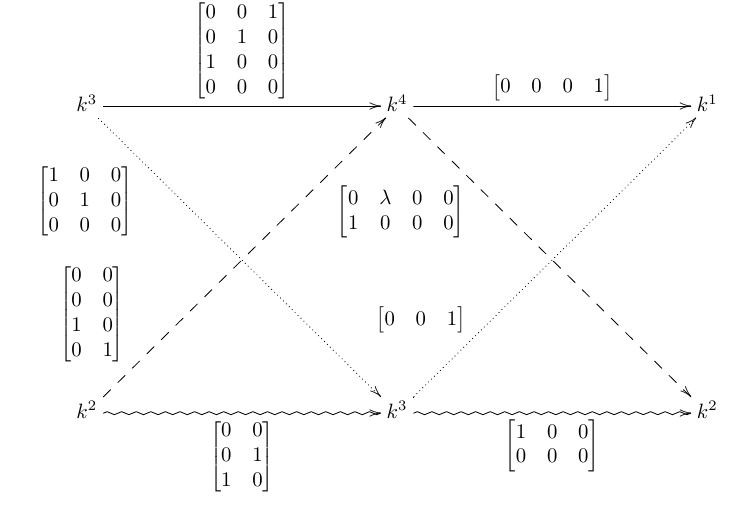}
\end{align*}

\end{example}

In \cite{Carroll2}, an explicit minimal projective resolution for each $M(\dd,\rr,\epsilon,\underline{\lambda})$ is constructed.  Using this resolution it is possible to show the following:
\begin{equation}
  \label{eq:ext1}
  \Ext^1_A(M(\dd,\rr,\epsilon,\underline{\lambda}), M(\dd, \rr,\epsilon,\underline{\lambda'}))=0
\end{equation}
whenever $\underline{\lambda},\underline{\lambda'}\in (k^*)^B$ are vectors that share no common coordinates, and that when $\Gamma(\dd,\rr,\epsilon)$ consists of a single band component,
\begin{equation}
  \label{eq:ext1band}
  \Ext^1_A(M(\dd,\rr,\epsilon, \underline{\lambda}), M(\dd,\rr,\epsilon,\underline{\lambda}))=k.
\end{equation}
In particular, these results can be exploited to show that
\begin{equation*}
  \module(A, \dd,\rr)=\overline{\bigcup_{\underline{\lambda}\in (k^*)^B}
    \GL(\dd)M(\dd,\rr,\epsilon, \underline{\lambda})}.
\end{equation*} 
When there are no band modules, then, $\module(A,\dd,\rr)$ is simply
the orbit closure of the rigid module $M(\dd,\rr,\epsilon)$.  This means that in order
to apply the results of Section \ref{sec:rational-invariants}, we need
only understand those irreducible components whose generic modules are
direct sums of band modules.

Given these results, we notice that the family $\GL(\dd)M(\dd,\rr,\epsilon,\underline{\lambda})$ is independent of the choice of $\epsilon$.  In the subsequent sections we will suppress this variable when referring to both the up-and-down graph and the up-and-down modules.  

It is important to note that the indecomposable direct summands of $M(\dd, \rr, \epsilon, \underline{\lambda})$ correspond to connected components of $\Gamma(\dd, \rr,\epsilon)$.  Thus, the graph $\Gamma(\dd,\rr,\epsilon)$ also gives the generic decomposition of $\module(A, \dd,\rr)$, i.e., the dimension vectors $\dd_i$ and rank functions $\rr_i$ such that
\begin{align*}
  \module(A, \dd,\rr) = \overline{\module(A,\dd_1,\rr_1)\oplus\dotsc\oplus \module(A, \dd_n, \rr_n)}.
\end{align*}
where $\module(A, \dd_i, \rr_i)$ are indecomposable irreducible components. In particular, in order to apply the reduction techniques from Proposition {~\ref{reduction-rational-inv-prop}} and Theorem \ref{theta-stable-decomp-thm}, it needs to be shown that if $\module(A, \dd_i, \rr_i)$ and $\module(A, \dd_j, \rr_i)$ are direct summands that are not orbit closures, and $\dd_i=\dd_j$, then $\rr_i=\rr_j$.  

\subsection{Regular Irreducible Components}
\label{subsec:reg-irr-cpts}

An irreducible component $\module(A, \dd, \rr)$ is called \emph{regular} if the generic module $M(\dd,\rr,\underline{\lambda})$ is a direct sum of bands. That is the only connected components of $\Gamma(\dd,\rr)$ are cycles.  We can characterize the ranks for which this is the case in the following proposition.

\begin{prop}{\cite[Section 5]{Carroll:Thesis}}
  \label{prop:udgraphconsequences}
  The irreducible component $\module(A, \dd,\rr)$ is regular if and
  only if the following two conditions hold:
  \begin{itemize}
  \item[I.] for any pair of consecutive monochromatic arrows $a, b$
    with $hb=ta=i$, $\rr(a)+\rr(b)=\dd(i)$;
  \item[II.] for any vertex $i$ incident to only one color, $\dd(i)=0$.
  \end{itemize}
\end{prop}

There are a number of important consequences that follow from this combinatorial requirement.

\begin{corollary}\label{cor:uniqueness}
If $\module(A, \dd)$ contains an irreducible regular component, then it is unique.  Furthermore, in this case $\langle\langle \dd,\dd\rangle\rangle=0$.  
\end{corollary}

\begin{proof}
  Suppose that $\module(A, \dd)$ contains an irreducible regular component.  We will show that $\rr$ is determined uniquely by the dimension vector and Proposition \ref{prop:udgraphconsequences}.  Given any color $s\in S$, let
  $a_{n_s}\cdot \dotsc\cdot a_1$ be the full path of the elements
  $c^{-1}(s)$.  Let $i_j=ha_j$ and $i_0=ta_1$.  Then
  $\rr(a_1)=\dd(i_0)$, and $\rr(a_{j+1})=\dd(i_j) -\rr(a_j)$, so the
  rank function $\rr$ for which $\module(A, \dd,\rr)$ is regular is
  uniquely determined by $\dd$.  As for the Euler form, one simply
  expresses the matrix of the form as a sum of contributions along
  each color, and then uses the fact that if $\module(A, \dd,\rr)$ is
  regular, then for any fully colored path as above, $\sum_{j=0}^{n_s}(-1)^j
  \dd(i_j)=0$.  Precise details can be found in \cite[Chapter 5.1]{Carroll:Thesis}.  
\end{proof}

It will be convenient to describe the projective resolution for
$X_{\underline{\lambda}}:=M(\dd,\rr,\underline{\lambda})$ when $\module(A, \dd,\rr)$ is a regular irreducible component (the
general case is handled in \cite[Section 3]{Carroll2}).  This will
allow us to derive a closed formula for the Schofield semi-invariants of
weight $\langle\langle \dd,-\rangle\rangle$.  Let us denote by $S^0$
and $S^1$ the set of sources and sinks of
$\Gamma=\Gamma(\dd,\rr)$, respectively.  For any $v_i^x\in
S^1$, we distinguish two distinct paths in $\Gamma$ terminating at
$v_i^x$, denoted $l^+(v_i^x)$ and $l^-(v_i^x)$.  The path
$l^\delta(v_i^x)$ is determined by the following two conditions:
\begin{itemize}
\item the tail $tl^\delta(v_i^x)$ is in $S^0$ and
\item $\epsilon(v_i^x,c(f_j^a))=\delta$ where $f^a_j$ is the arrow in
  $l^\delta(v_i^x)$ incident to $v_i^x$.  
\end{itemize}
Notice that this indeed defines two distinct paths, since in $\Gamma$
each vertex is incident to precisely two colored arrows whose signs
under $\epsilon$ are different.  We now consider the modules
\begin{align*}
  P_0&=\bigoplus\limits_{v_i^x\in S^0} P(v_i^x)\\
  P_1&=\bigoplus\limits_{v_i^x\in S^1} P(v_i^x)
\end{align*}
with a map $F: P_1\rightarrow P_0$ defined in the following way.  Let $F|_{(i,x)}: P(v_i^x) \rightarrow P(tl^+(v_i^x))\oplus P(tl^-(v_i^x))$ be the map given by
\begin{equation*}
  P(v_i^x)\xrightarrow{\begin{bmatrix} \lambda_b l^+(v_i^x) \\
      -l^-(v_i^x)\end{bmatrix}} P(tl^+(v_i^x))\oplus
  P(tl^-(v_i^x))\end{equation*}
if there is a band $b\in B$ with $\Theta(b)=v_i^x$, or
\begin{equation*}
    P(v_i^x)\xrightarrow{\begin{bmatrix} l^+(v_i^x) \\
      -l^-(v_i^x)\end{bmatrix}} P(tl^+(v_i^x))\oplus
  P(tl^-(v_i^x))\end{equation*}
otherwise.  Following $F|_{(i,x)}$ by the inclusion gives a map
$P(v_i^x)\rightarrow P_0$ which we denote with the same symbol.  Then
$F$ is defined to be the map $\sum_{v_i^x\in S^1} F|_{(i,x)}$.

In \cite{Carroll2} it was shown that $P_1\rightarrow P_0 \rightarrow
X_{\underline{\lambda}} \rightarrow 0$ is a minimal projective resolution of
$X_{\underline{\lambda}}$.  In particular, $X_{\underline{\lambda}}$ has projective dimension one
for all $\underline{\lambda} \in (k^*)^B$.  One important consequence is the
following.

\begin{prop}
  \label{prop:schurbands}
If $M(\dd,\rr,\lambda)$ is a generic indecomposable (regular) module, then it is Schur.
\end{prop}
\begin{proof}
Suppose $M=M(\dd,\rr,\lambda)$ is indecomposable regular.  As remarked, its projective dimension is 1, so, $\langle \langle \dd, \dd\rangle \rangle =\dim_k \Hom_A(M,M) - \dim_k \Ext^1_A(M,M)$.  Putting this together with Corollary \ref{cor:uniqueness} and Equation \ref{eq:ext1band}, this yields:
    \begin{align*}
    0=\langle\langle \dd,\dd\rangle\rangle &= \dim_k \Hom_A(M,M) +1.
    \end{align*}
The proposition follows immediately.
\end{proof}
 
\begin{corollary}\label{cor:schurbands}
Let $\module(A, \dd,\rr)$ be an indecomposable irreducible regular component of $\module(A, \dd)$.  Then the following hold.
 \begin{itemize}
 \item[1.] The field of rational invariants $k(\module(A, \dd,\rr))^{\GL(\dd)}$ is purely transcendental of degree 1;
 \item[2.] If $\theta$ is a weight such that
   $\module(A,\dd,\rr)_\theta^s\neq \emptyset$, then $\M(\module(A, \dd,\rr))^{ss}_\theta\cong\mathbb{P}^1$.
\end{itemize}
\end{corollary}
 
\begin{proof}
Part 1 follows immediately from proposition above and Proposition
\ref{prop:tamerationality}{(3)}.  As for part 2, we have already
mentioned that $\module(A, \dd, \rr)$ is normal, so $\M(\module(A,
\dd,\rr))^{ss}_\theta$ is either a point or $\mathbb{P}^1$ by
Proposition \ref{prop:tamerationality}{(2)}.  But from part (1) and
Remark \ref{rmk-fine-moduli-rational-inv}, the moduli space is not a point.  
\end{proof}

It turns out that we can actually extend the second part of the previous result to the case of triangular string algebras:

\begin{corollary}\label{cor:string-moduli} Let $A=kQ/I$ be a triangular string algebra, $\dd$ a generic root of $A$, and $C \subseteq \module(A,\dd)$ an indecomposable irreducible component. If $\theta \in \ZZ^{Q_0}$ is an integral weight such that $C^s_{\theta} \neq \emptyset$ then $\M(C)^{ss}_{\theta}$ is either a point or $\mathbb P^1$. 
\end{corollary}

\begin{proof} It has been proved in \cite[Proposition 2.8]{Carroll1} that there exists a coloring $c$ of $Q$ such that $I_c\subseteq I$ and $A'=kQ/I_c$ is a colored gentle algebra. Hence, there exists an irreducible component $C' \subseteq(A',\dd)$ that contains  $C$. 

If $C$ is an orbit closure then the moduli space in question is just a point. The other case left is when $\M(C)^{ss}_{\theta}$ is a rational projective curve (see Proposition \ref{prop:tamerationality}{(2)}). Then, $\M(C')^{ss}_{\theta} \simeq \mathbb \PP^1$ by Corollary \ref{cor:schurbands}.

Now, let $\pi:C^{ss}_{\theta} \to \M(C)^{ss}_{\theta}$ and $\pi':C'^{ss}_{\theta} \to \M(C')^{ss}_{\theta}$ be the quotient morphisms for the actions of $\PGL(\dd)$ on $C^{ss}_{\theta}$ and $C'^{ss}_{\theta}$, respectively. Now, consider the $\PGL(\dd)$-invariant morphism $\varphi:C^{ss}_{\theta}\to \M(C')^{ss}_{\theta}$ that sends $M$ to $\pi'(M)$. From the universal property of GIT quotients, we know that there exists a morphism $f:\M(C)^{ss}_{\theta} \to \M(C')^{ss}_{\theta}$ such that $f \circ \pi=\varphi$. 

It is easy to see that $f$ is injective: Indeed, let $M_1,M_2 \in C^{ss}_{\theta}$ be so that $f(\pi(M_1))=f(\pi(M_2))$. Then, $\pi'(M_1)=\pi'(M_2)$, which is equivalent to $\overline{\GL(\dd)M_1}\cap \overline{\GL(\dd)M_2}\cap C'^{ss}_{\theta} \neq \emptyset$. Observe that $C^{ss}_{\theta}$ is closed in $C'^{ss}_{\theta}$ and hence $\overline{\GL(\dd)M_i}\cap C'^{ss}_{\theta}$, $i=1,2$, are contained in $C^{ss}_{\theta}$. So, $\overline{\GL(\dd)M_1}\cap \overline{\GL(\dd)M_2}\cap C^{ss}_{\theta} \neq \emptyset$, which is equivalent to $\pi(M_1)=\pi(M_2)$.

Moreover, it is clear now that $f$ is surjective since otherwise $\Ima f$ would be just a point. In conclusion, $f$ is a bijective morphism whose target variety is normal, and hence it has to be an isomorphism by (a consequence of) the Zariski's Main Theorem (in characteristic zero).
\end{proof}

\section{Proofs of the main results} \label{proofs-sec} We first prove Theorem \ref{mainthm-1}, and then proceed with the proof of our constructive Theorem \ref{mainthm-2}.
 
\begin{proof}[Proof of Theorem \ref{mainthm-1}]  $(1)$ Using Corollary  \ref{cor:uniqueness}, we know that we can write the generic decomposition of $C$ as:
$$
C=\overline{\bigoplus_{i=1}^n C_i^{\oplus m_i}\oplus \bigoplus_{j=n+1}^l C_j}
$$  
where $C_i \subseteq \rep(\Lambda,\dd_i)$, $1 \leq i \leq l$, are indecomposable irreducible regular components, $C_{n+1}, \dots, C_l$ are orbit closures, $m_1,\ldots, m_n$ are positive integers, and $\dd_i\neq \dd_j, \forall 1 \leq i\neq j\leq n$. (Of course, we do allow $n=0$ or $l=n$.)

But note that for each $1 \leq i \leq n$, $k(C_i)^{\GL(\dd_i)}\simeq k(t)$ by Corollary \ref{cor:schurbands}{(1)}. Applying now Proposition \ref{reduction-rational-inv-prop}, we obtain that $k(C)^{\GL(\dd)} \simeq k(t_1,\ldots, t_N)$ where $N=\sum_{i=1}^n m_i$.

$(2)$ First of all, any $\theta$-semi-stable irreducible component $C \subseteq \module(A,\dd)$ is $\theta$-well-behaved. This follows immediately from the uniqueness property in Corollary \ref{cor:uniqueness}.

Now, let us assume that $C$ is a $\theta$-semi-stable irreducible regular component whose $\theta$-stable decomposition is of the form:
$$
C=C_1\pp \ldots \pp C_l
$$
with $C_1, \ldots, C_l$ indecomposable irreducible regular components. Using Corollary \ref{cor:uniqueness} again, we can write the $\theta$-stable decomposition of $C$ as:
$$C=l_1\cdot C_1 \pp \ldots \pp l_m\cdot C_m$$ where $C_i \subseteq \module(A,\dd_i)$, $1 \leq i \leq m$, are $\theta$-stable irreducible regular components, and $\dd_i\neq \dd_j$ for all $1 \leq i\neq j \leq m$.

Moreover, if $C' \subseteq \module(A,\dd)$ is an irreducible component containing $\overline{\bigoplus_{i=1}^m C_i^{\oplus l_i}}$ then $C'$ contains a regular module, and so $C=C'$ by Corollary \ref{cor:uniqueness}. At this point, we can apply Theorem \ref{theta-stable-decomp-thm} and Corollary \ref{cor:schurbands}{(2)} and deduce that $\M(C)^{ss}_{\theta} \simeq \prod_{i=1}^m \mathbb P^{l_i}$. 
\end{proof}

\vspace{35pt}

To prove Theorem \ref{mainthm-2}, we will use the projective resolution from Section \ref{subsec:reg-irr-cpts}. But first let us briefly recall the construction of the so-called generalized Schofield semi-invariants on module varieties following Derksen and Weyman \cite{DW4}, and Domokos \cite{Domo}.

Let $X$ be an $A$-module and $P_1\xrightarrow{F} P_0 \rightarrow X
\rightarrow 0$ be a fixed minimal projective presentation of $X$ in
$\module(A)$.  Let $\theta^X\in \ZZ^{Q_0}$ be the integral weight
defined so that $\theta^X(v)$ is the multiplicity of $P_v$ in $P_0$
minus the multiplicity of $P_v$ in $P_1$ for all $v\in Q_0$.  For an
arbitrary $A$-module $M$, we have
\begin{align*}
  \theta^X(\ddim M) &= \dim_k \Hom_A(P_0,M)-\dim_k \Hom_A(P_1,M)\\
  &=\dim_k \Hom_A(X, M)-\dim_k \Hom_A(M, \tau X).
\end{align*}
Here, $\tau X$ is the Auslander-Reiten translation of $X$ given by
$D(\coker f^t)$ where $(-)^t=\Hom_A(-,A)$ and $D=\Hom_k(-,k)$ (for
more details, see \cite[\S IV.2]{AS-SI-SK}).  Notice that if $\pdim
X\leq 1$ then
$$\theta^X(\ddim M) = \dim_k \Hom_A(X, M)-\dim_k \Ext^1_A(X,
M)=\langle\langle \ddim X, \ddim M\rangle\rangle.$$

Before we continue our discussion on Schofield semi-invariants we mention the following example of a weight that satisfies the conditions in Theorem \ref{mainthm-1}{(2)}:

\begin{example} Let $C\subseteq \module(A,\dd)$ be an irreducible regular component. Then, it is generic decomposition is of the form:
$$
C=\overline{\bigoplus_{i=1}^mC_i^{l_i}},$$
where $C_i \subseteq \module(A,\dd_i)$, $1 \leq i \leq m$, are indecomposable irreducible regular components and $\dd_i \neq \dd_j, \forall 1 \leq i\neq j \leq m$. Notice also that $\langle \langle \dd_i, \dd_i \rangle \rangle=0$ for all $1 \leq i, j \leq m$ (for details, see Lemma \ref{lem:arbitrary-regular}).

We claim that the weight $\theta:=\langle \langle \dd, \cdot \rangle \rangle$ satisfies the conditions in Theorem \ref{mainthm-1}{(2)}. Specifically, we claim that the $\theta$-stable decomposition of $C$ is precisely 
$$
C=l_1\cdot C_1 \pp \ldots \pp l_m\cdot C_m.$$ 
For this, it is clearly enough to show that $C_{i, \theta}^s \neq \emptyset$ for each $1 \leq i \leq m$.  Let $M_i \in C_i$ be a Schur homogeneous module (e.g. choose $M_i$ to be any indecomposable band module in $C_i$). Let us check that $M_i \in C_{i, \theta}^s$. Set $M:=\bigoplus_{i=1}^m M_i^{l_i}$ and note that $\pdim M \leq 1$ and $\tau M \simeq M$ since each $M_i$ has these two properties. From the discussion above, we have that $\theta=\theta^M$ and so
$$
\theta(\ddim M'_i)=\dim_K \Hom_A(M,M'_i)-\dim_K \Hom_A(M'_i,M),
$$
for any proper submodule $M'_i <M_i$. Using that $M_i$ is Schur, $\Hom_A(M_j,M_l)=0$ for all $j \neq l$, and $M'_i \neq M_i$, it is easy to see that $\Hom_A(M,M'_i)=0$ and hence $$\theta(\ddim M'_i)=-\dim_K\Hom_A(M'_i,M)<0.$$ 

This shows that $C$ has the desired $\theta$-stable decomposition. We can now apply Theorem \ref{mainthm-1}{(2)} and conclude that $\M(C)^{ss}_{\theta}\simeq \prod_{i=1}^m \mathbb P^{l_i}$.
\end{example}

Let $\dd\in\ZZ^{Q_0}$ be a dimension vector of $A$ such that
$\theta^X(\dd)=0$. Then for any module $M\in \module(A, \dd)$,
$\dim_k\Hom_A(P_0,M) = \dim_k \Hom_A(P_1,M)$ and hence the linear
map \[\Hom_A(F, M): \Hom_A(P_0,M)\rightarrow \Hom_A(P_1, M)\] can be
viewed as a square matrix.  We can therefore define
\[\overline{c}^X:\module(A,\dd)\rightarrow k, \qquad
\overline{c}^X(M)=\det d^X(M).\]
It is easy to see that $\overline{c}^X$ is a semi-invariant of weight
$\theta^X$. Moreover, any other choice of a minimal projective
presentation of $X$ leads to the same semi-invariant, up to a non-zero
scalar. We call $\overline{c}^X$ a \emph{generalized Schofield semi-invariant}.

We can now apply this setup to triangular gentle algebras $A=kQ/I_c$, when $X=M(\dd,\rr,\lambda)$, a generic indecomposable regular module.  Let $$\module(A, \dd,\rr)=\overline{\bigcup_{\lambda\in k^*} \GL(\dd)\cdot M(\dd,\rr,\lambda)}$$ be an indecomposable irreducible
regular component.  Evidently, $M(\dd,\rr,\lambda)$ has projective dimension 1,
and from Corollary \ref{cor:uniqueness} we have that $\langle\langle
\dd,\dd\rangle\rangle =0$.  Therefore, $\theta^{M(\dd,\rr,\lambda)}(\dd) = 0$
and $\overline{c}^{M(\dd,\rr,\lambda)}(M)=\det \Hom_A(F, M)$ for all $M\in
\module(A, \dd,\rr)$ where $P_1\xrightarrow{F} P_0 \rightarrow
M(\dd,\rr,\lambda) \rightarrow 0$ is the minimal projective resolution described in Section \ref{subsec:reg-irr-cpts}. It can be shown that for any
$\mu\in k^*$, the map $\Hom_A(F, M(\dd,\rr,\mu))$ is a block
diagonal matrix with diagonal blocks $D_0, D_1,\dotsc, D_t$ where for
$i=1,\dotsc, t$ the $D_i$ are invertible upper triangular matrices
whose only non-zero entries are from the set $\{\pm 1, \pm \lambda,
\pm \mu\}$, and $D_0$ is a matrix of the form:
\begin{equation*}
  D_0 = \begin{bmatrix} -\lambda & \mu& 0 & 0 & \dotsc & 0 \\ 0 & \pm 1 & \pm 1 & 0 & \dotsc & 0\\ 0 & 0 & \pm 1 & \pm 1 & \dotsc& 0\\ \vdots & &\vdots& &\ddots & \vdots \\ \pm 1 & 0&0 &&\dotsc& \pm 1\end{bmatrix}
\end{equation*}
such that in each row there is precisely one positive entry and one negative entry.
Expanding the determinant of the matrix yields $\det \Hom_A(F,
M(\dd,\rr,\lambda)) = \lambda^p \mu^l(\lambda-\mu)$.
Summarizing, we have the following (see \cite[Section 3]{Carroll2}).

\begin{prop} \label{prop:semiinvariantband}
Let $\module(A, \dd,\rr)=\overline{\bigcup_{\lambda\in k^*} \GL(d)\cdot M(\dd,\rr,\lambda)}$ be an indecomposable irreducible regular component of $\module(A, \dd)$.
Then \begin{equation}\label{eq:cvformula}
  \overline{c}^{M(\dd,\rr,\lambda)}(M(\dd,\rr,\mu))=\lambda^p\mu^l(\lambda-\mu)
\end{equation} for some constants $p=p(\dd,\rr)$, $l=l(\dd,\rr)$.
\end{prop}

We now extend this formula to the general case, i.e., when $\module(A,\dd,\rr)$ is an irreducible regular component (not necessarily indecomposable).  
\begin{lemma}\label{lem:arbitrary-regular}
Suppose $\module(A, \dd, \rr) =\overline{\bigoplus_{i=1}^n\module(A, \dd_i,\rr_i)^{m_i}}$ is the generic decomposition of the irreducible regular component $\module(A,\dd, \rr)$, then for each $i=1,\dotsc, n$, $\langle\langle \dd_i, \dd \rangle\rangle=0$. 
\end{lemma}

\begin{proof}
Recall from equation \ref{eq:ext1} that $\Ext^1_A(M(\dd, \rr, \underline{\lambda}), M(\dd, \rr,\underline{\lambda'}))=0$ when $\underline{\lambda}, \underline{\lambda'}$ share no common components.  In particular, $\Ext^1_A(M(\dd_i,\rr_i,\lambda_i), M(\dd_j, \rr_j,\lambda_j))=0$ for $i\neq j$ (again, assuming $\lambda_i\neq \lambda_j$).  Moreover, by definition $\module(A, \dd_i,\rr_i)$ is a regular irreducible component, so $\langle\langle \dd_i, \dd_i\rangle\rangle=0$. 

 Let $\{\lambda_j; j=1,\dotsc, n\}$ be a collection of pairwise distinct elements in $k^*$, and abbreviate $M(\dd_j,\rr_j, \lambda_j)$ by $X^{(j)}$.  Then for any given $i$, $\langle\langle \dd_i, \dd\rangle \rangle \geq 0$ since 
\begin{align*}
\langle\langle \dd_i, \dd \rangle\rangle &= \sum\limits_{j=1}^n
m_j \langle\langle \dd_i, \dd_j\rangle\rangle \\ &=
\sum\limits_{j\neq i} m_j \langle\langle\dd_i, \dd_j\rangle\rangle\\
&= \sum\limits_{j\neq i} m_j (\dim \Hom(X^{(i)}, X^{(j)}) -  \dim
\Ext^1_A(X^{(i)}, X^{(j)})\\
&= \sum\limits_{j\neq i} m_j \dim \Hom_A(X^{(i)}, X^{(j)}) \geq 0.
\end{align*}
However, from Corollary {~\ref{cor:uniqueness}}, $0=\langle\langle
\dd, \dd \rangle \rangle = \sum_{i=1}^n m_i\langle \langle \dd_i,
\dd\rangle\rangle$, and so $\langle \langle \dd_i, \dd \rangle \rangle
=0$ since $m_i\geq 1$.  
\end{proof}

\begin{remark}\label{rk:generalschofield} Given this proposition, the generalized Schofield invariant $\overline{c}^{M(\dd_i,\rr_i,\lambda)}$ is a non-zero semi-invariant function on the regular component $\module(A, \dd,\rr)$.
\end{remark}

Therefore, with some extra considerations (namely, the block diagonal form of $\Hom_A(F, M)$), Proposition {~\ref{prop:semiinvariantband}} can be applied more generally to irreducible components.

\begin{prop}
  \label{prop:arbitrary-bands}
Suppose again that $\module(A, \dd, \rr) =\overline{\bigoplus_{i=1}^n\module(A, \dd_i,\rr_i)^{\oplus m_i}}$ is the generic decomposition of the irreducible regular component $\module(A, \dd, \rr)$ and for $\underline{\mu}=(\mu(i,j); i=1,\dotsc, n,\ j=1,\dotsc, m_i,\ \mu(i,j)\in k^*)$ consider the generic module
  \begin{equation*}
    X_{\underline{\mu}}=\bigoplus\limits_{i=1}^n \left(\bigoplus\limits_{j=1}^{m_i} M(\dd_i,\rr_i, \mu(i,j))\right)
  \end{equation*}
  Then for any $\lambda \in k^*$,
  \begin{equation}
    \label{eq:arbitrary-band}
    \overline{c}^{M(\dd_i, \rr_i, \lambda)}(X_{\underline{\mu}})=
    \left[\prod\limits_{j=1}^{m_i} \lambda^{p_i} \mu(i,j)^{l_i}
      (\lambda-\mu(i,j))\right] \cdot \left[\prod\limits_{i'\neq
        i}\prod\limits_{j=1}^{m_{i'}} \lambda^{p_{i'}} \mu(i',j)^{l_{i'}}\right]
  \end{equation}
for some positive integers $p_i, l_i$.
\end{prop}

This allows us to prove our main constructive theorem:

\begin{proof}[Proof of Theorem \ref{mainthm-2}] 
We need only exhibit an open set $U \subset \module(A, \dd, \rr)$ on which the
functions separate orbits.  In this case,
\cite[Lemma 2.1]{PopVin} implies that
$$k(\module(A, \dd, \rr))^{\GL(\dd)} \cong
k\left(\frac{\overline{c}^{M(\dd_i,\rr_i,\lambda(i,j))}}{\overline{c}^{M(\dd_i,\rr_i,\lambda(i,j+1))}}
  ; i=1,\dotsc, n,\ j=0,\dotsc, m_i-1\right).$$  There are clearly
$N=\sum m_i$ such functions, from which we can choose a subset that
forms a transcendental basis.  By Theorem {~\ref{mainthm-1}}, $N$ is
the transcendence degree of $k(\module(A, \dd, \rr))^{\GL(\dd)}$ over
$k$, but taking any fewer elements from the set will yield a smaller
transcendence degree, a contradiction.

Now we prove the first claim. Notice that as the image of the dominant map
  $\GL(\dd)\times (k^*)^B \rightarrow \module(A, \dd, \rr)$, the set
  $\bigcup_{\mu \in (k^*)^B} \GL(\dd)\cdot M(\dd, \rr,\mu)$ is
  constructible, so it contains an open dense set $U'$ of $\module(A,
  \dd,\rr)$.  Take $U$ to be the open set 
  $U=\bigcap_{i, j}\{X\in
  \module(A, \dd, \rr) \mid
  \overline{c}^{M(\dd_i,\rr_i,\lambda(i,j))}\neq 0\}\cap U'$, and let
  $X_{\underline{\mu}}, X_{\underline{\nu}}$ be two generic modules in $U$ such
  that for each $i=1,\dotsc, n,\ j=0,\dotsc, m_i-1$,
  \begin{equation*}
    \frac{\overline{c}^{M(\dd_i,\rr_i,\lambda(i,j))}(X_{\underline{\mu}})}{\overline{c}^{M(\dd_i,\rr_i, \lambda(i,j+1))}(X_{\underline{\mu}})}=\frac{\overline{c}^{M(\dd_i,\rr_i,\lambda(i,j))}(X_{\underline{\nu}})}{\overline{c}^{M(\dd_i,\rr_i, \lambda(i,j+1))}(X_{\underline{\nu}})}.
  \end{equation*}
Expanding these by applying Equation {~\ref{eq:arbitrary-band}}, for each, we have
  \begin{align*}
      \frac{\left[ \prod\limits_{j'=1}^{m_i}
          \lambda(i,j)^{p_i}\mu(i,j')^{l_i}(\lambda(i,j)-\mu(i,j'))\right]\cdot
        \left[\prod\limits_{i'\neq i} \prod\limits_{j'=1}^{m_{i'}} \lambda(i,j)^{p_{i'}}\mu(i',j')^{l_{i'}}\right]}{\left[\prod\limits_{j'=1}^{m_i}
          \lambda(i,j+1)^{p_i}\mu(i,j')^{l_i}(\lambda(i,j+1)-\mu(i,j'))\right]\cdot
        \left[\prod\limits_{i'\neq i} \prod\limits_{j'=1}^{m_{i'}}
          \lambda(i,j+1)^{p_{i'}}\mu(i',j')^{l_{i'}}\right]}\\ \qquad=  \frac{\left[ \prod\limits_{j'=1}^{m_i}
          \lambda(i,j)^{p_i}\nu(i,j')^{l_i}(\lambda(i,j)-\mu(i,j'))\right]\cdot
        \left[\prod\limits_{i'\neq i} \prod\limits_{j'=1}^{m_{i'}} \lambda(i,j)^{p_{i'}}\nu(i',j')^{l_{i'}}\right]}{\left[\prod\limits_{j'=1}^{m_i}
          \lambda(i,j+1)^{p_i}\nu(i,j')^{l_i}(\lambda(i,j+1)-\nu(i,j'))\right]\cdot
        \left[\prod\limits_{i'\neq i} \prod\limits_{j'=1}^{m_{i'}} \lambda(i,j+1)^{p_{i'}}\nu(i',j')^{l_{i'}}\right]}
  \end{align*}
  This is equivalent to the existence of a constant $c_i$ for each
  $i=1,\dotsc, n$ with 
  \begin{equation*}
    \frac{\prod\limits_{j'=1}^{m_i}
      (\lambda(i,j)-\mu(i,j'))}{\prod\limits_{j'=1}^{m_i}
      (\lambda(i,j)-\nu(i,j'))} = c_i
  \end{equation*}
  for each $j=0,\dotsc, n$.  In particular, for each $i$, $\prod\limits_{j'=1}^{m_i}
  (x-\mu(i,j')) - c_i \prod\limits_{j'=1}^{m_i} (x-\nu(i,j'))$ is a
  degree $n$ polynomial with $n+1$ distinct roots (by selection of the
  $\lambda(i, j)$), and so for all $i$, $c_i=1$.  Thus, for each $i$ there is
  a permutation $\sigma^{(i)}$ of $\{1,\dotsc, m_i\}$ such that
  $\mu(i,j')=\mu(i,\sigma^{(i)}(j'))$.  In particular, $X_{\underline{\mu}} \cong X_{\underline{\nu}}$.
\end{proof}

\begin{example} Finally, we point out some pathologies that show the necessity of the conditions of Proposition {~\ref{reduction-rational-inv-prop}}.
\label{ex:pathologies}
\begin{enumerate}
\renewcommand{\theenumi}{\roman{enumi}}

\item Consider the gentle algebra $A(2)=kQ/I_c$ where the bound quiver and coloring are indicated below:
\begin{equation*}
      \xymatrix{1 \ar@/^/[r]^{a_1} \ar@{-->}@/_/[r]_{b_1} & 2 \ar@/^/[r]^{a_2}
        \ar@{-->}@/_/[r]_{b_2} & 3}.
\end{equation*}
Let $\dd=(2,2,2)$ and $\rr(a_i)=\rr(b_i)=1$ for $i=1,2$. The generic module in this case is the module
\begin{equation*}
      \xymatrix@R=.1cm{ \circ \ar[r] & \circ \ar@{-->}[r] & \circ\\
        \circ \ar@{-->}[r] & \circ \ar[r] & \circ}.
\end{equation*}
Let $\rr'$ be the rank function $\rr'(a_1)=\rr'(b_2)=1$ and $\rr'(a_2)=\rr'(b_1)=0$, and $\rr''$ be the rank function $\rr''(a_2)=\rr''(b_1) = 1$ and $\rr''(a_1)=\rr''(b_2) =0$. Then the generic decomposition of $\module(A, \dd, \rr)$ is $\overline{\module(A, (1,1,1),\rr') \oplus \module(A, (1,1,1), \rr'')}$. Thus, $\module(A, \dd,\rr)$ decomposes generically into the direct sum of two distinct orbit closures of the same dimension vector.  This is in contrast to the situation in which $\module(A, \dd, \rr)$ is a regular component (see Corollary {~\ref{cor:uniqueness}}).

\item Indecomposable generic modules are not necessarily Schur. Consider the representation space of the bound quiver with coloring, dimension vector, and rank sequence as depicted in the diagram:
\begin{equation*}
          \xymatrix{*++[F]{1} \ar[dr]_1 \ar@{-->}[r]^1 &
            *++[F]{2} \ar@{-->}[r]^1 & *++[F]{1} \\ & *++[F]{1}
            \ar@{..>}[ur]_1 & }
\end{equation*}
The generic module $X$ in this module variety takes the form
\begin{equation*}
         \xymatrix@R=.1cm{\circ \ar@{-->}[r] \ar[ddddr] & e_1 &\circ  \\ & e_2
            \ar@{-->}[ur] &  \\ && \\ && \\  & \circ \ar@{..>}[uuuur] &}
\end{equation*}
where $e_1, e_2$ is a basis for the 2-dimensional vector space in the top row.  There is a non-trivial endomorphism $X\rightarrow X$ given on basis elements by $e_2\mapsto e_1$ and all other basis elements are sent to 0.
\item It is possible that in $\module(A, \dd)$ there are two distinct indecomposable irreducible components, one of which is regular.  Consider the module variety of
\begin{equation*}
          \xymatrix{& \circ \ar@{~>}[dr] & \circ \ar@{-->}[dr] & \\ \circ \ar[r]_{a_1}
            \ar@{-->}[ur] & \circ \ar[r]_{a_2} \ar@{..>}[ur] & \circ\ar[r]_{a_3} & \circ} 
\end{equation*}
with dimension vector $\dd=(1,1,\dotsc, 1)$, and the following two rank functions: $\rr$ defined by $\rr(a_2)=0$ and all other ranks are one, and $\rr'$ defined by $\rr'(a_1)=\rr'(a_3)=0$ and all other ranks are one.  It is clear that these are maximal with respect to $\dd$.  The generic modules are given below:
\begin{equation*}
          \xymatrix{& \circ \ar@{~>}[dr] & \circ \ar@{-->}[dr] & \\ \circ \ar[r]
            \ar@{-->}[ur] & \ar@{..>}[ur] & \circ\ar[r]& \circ} \qquad
          \xymatrix{& \circ \ar@{~>}[dr] & \circ \ar@{-->}[dr] & \\ \circ
            \ar@{-->}[ur] & \circ \ar[r] \ar@{..>}[ur] & \circ & \circ} 
\end{equation*}
Notice that the generic module in $\module(\dd,\rr)$ is an indecomposable band, and that in $\module(\dd,\rr')$ is an indecomposable string.  

\item Consider the triangular string algebra $A=kQ/I$ where $Q$ is the quiver
  \begin{align*}
   \xymatrix{1 \ar[dd]_c \ar[dr]^a && 2\ar[dd]^d \ar[dl]_b\\ & 3 \ar[dr]_f \ar[dl]^e& \\ 4 & & 5}
  \end{align*}
 
and $I$ is the ideal generated by $ea, eb, fa, fb$. We call $A$ the \emph{butterfly} string algebra. Notice that $A$ is not a colored gentle algebra. 

It has been proved in \cite[Section 4.1]{CC12} that there is only one indecomposable irreducible regular component $C$, namely the one corresponding to the band $B=c^{-1}ef^{-1}db^{-1}a$ and dimension vector $\dd=(1,1,2,1,1)$, and $C$ is not a Schur component. Nonetheless, it has been proved in \cite[Theorem 4.1]{CC12} that for any weight $\theta \in \ZZ^5$ with $C^{ss}_{\theta} \neq \emptyset$, the moduli space $\M(C)^{ss}_{\theta}$ is just a point.
\end{enumerate}
\end{example}

\section{Concluding remarks}\label{remarks-sec}
We conclude this section by putting these results in the the broader context of a program aimed at characterizing  tame algebras and parametrizing modules over such algebras  using invariant-theoretic techniques.  It was Jerzy Weyman in private communication who first made the following conjecture.

\begin{conjecture}[\textbf{Weyman}]\label{conj:Weyman}
Let $(Q, I)$ be a bound quiver, and $A=kQ/I$ its bound quiver algebra. Then the following are equivalent:
\begin{enumerate}
\item $A$ is of tame representation type;
\item for any irreducible component $C\subset \module(A, \dd)$ and any weight $\theta$ such that $C_\theta^{ss}\neq \emptyset$, $\M(C)_\theta^{ss}$ is a product of projective spaces.
\end{enumerate}
\end{conjecture}

The implication $(2) \Longrightarrow (1)$ of this conjecture has been verified for the class of quasi-tilted algebras and of strongly simply connected algebras (see \cite{CC10}). 

The other implication $(1)\Longrightarrow (2)$ of the conjecture above has been verified in \cite{CC10} for quasi-tilted algebras when the irreducible components involved are isotropic. Since gentle algebras are tame, Theorem \ref{mainthm-1}{(2)} proves this implication over regular irreducible components for triangular gentle algebras. In order to prove it over arbitrary irreducible components, one would need to have a better understanding of the $\theta$-stable decomposition of irreducible components of module varieties. We plan to address this in future work.

A particular case of Conjecture \ref{conj:Weyman} is:

\begin{conjecture}\label{conj-weak:Weyman} Let $A$ be a tame algebra, $\dd$ a generic root of $A$, and $C \subseteq \module(A,\dd)$ an indecomposable irreducible component. Then, for any weight $\theta$ for which $C^{ss}_{\theta} \neq \emptyset$, the moduli space $\M(C)^{ss}_{\theta}$ is either a point or $\mathbb P^1$. 
\end{conjecture}
   
Notice that Proposition \ref{prop:tamerationality} and Corollary \ref{cor:string-moduli} offer evidence for this conjecture. In fact, a proof of it for triangular string algebras would follow from Corollary \ref{cor:string-moduli} and a positive resolution to the following question:

\begin{question} Let $A$ be a tame algebra, $\dd$ a generic root, and $C \subseteq \module(A,\dd)$ an indecomposable irreducible component. If $\dim M(C)^{ss}_{\theta} =1$, is it true that $C^s_{\theta} \neq \emptyset$?
\end{question}

This question has a positive answer when $A$ is a tame quasi-tilted algebra, or a triangular gentle algebra, or the butterfly string algebra from Example \ref{ex:pathologies}.

\vspace{25pt}
The birational analogue of Conjecture \ref{conj:Weyman} would be the following, suggested by the second author.

\begin{conjecture}\label{conj:Chindris}
Let $(Q, I)$ be a bound quiver, and $A=kQ/I$ its bound quiver algebra. 
\begin{itemize}
\item [(a)] The following are equivalent:
\begin{enumerate}
\item $A$ is of tame representation type;
\item for every irreducible indecomposable component $C\subset \module(A, \dd)$, $k(C)^{\GL(\dd)}$ is a rational field of transcendence degree at most one, that is $k$ or $k(t)$;
\end{enumerate}

\item [(b)] Assume that $A$ is a tame algebra. Then, for any irreducible component $C \subseteq \module(A,\dd)$, $k(C)^{\GL(\dd)}$ is a rational field over $k$.
\end{itemize}
\end{conjecture}

Note that condition $(2)$ of the conjecture above simply says that for an indecomposable irreducible component $C \subseteq \module(A,\dd)$ the rational quotient $C/\GL(\dd)$ is either a point or just $\mathbb P^1$, and this is very much in sync with the philosophy behind the tameness of an algebra. 

Conjecture \ref{conj:Chindris} has been verified for the class of quasi-tilted algebras (see \cite{CC9,CC10}). Moreover, our Theorem \ref{mainthm-1} is a verification of Conjecture \ref{conj:Chindris} for the class of triangular gentle algebras. 

The implication $(2) \Longrightarrow (1)$ has also been verified for the class of strongly simply connected algebras (see \cite{CC10}). We plan to use the methods of this paper to settle this conjecture for the class of strongly simply connected algebras.

A separate and still difficult problem in general is the determination of a transcendental basis of $k(C)^{\GL(\dd)}$. This was achieved by Ringel in the tame hereditary case \cite{R4}, but in general such a basis is unknown. Our Theorem \ref{mainthm-2} gives such a basis in the case when $A$ is a triangular gentle algebra and $C$ is an irreducible regular component.

\bibliography{biblio}\label{biblio-sec}

\begin{thebibliography}{10}

\bibitem{ABCH}
D.~{Arcara}, A.~{Bertram}, I.~{Coskun}, and J.~{Huizenga}.
\newblock {The Minimal Model Program for the Hilbert Scheme of Points on P\^{}2
  and Bridgeland Stability}.
\newblock {\em ArXiv e-prints}, March 2012.

\bibitem{AS-SI-SK}
I.~Assem, D.~Simson, and A.~Skowro{\'n}ski.
\newblock {\em Elements of the representation theory of associative algebras.
  {V}ol. 1: {T}echniques of representation theory}, volume~65 of {\em London
  Mathematical Society Student Texts}.
\newblock Cambridge University Press, Cambridge, 2006.

\bibitem{ABCP}
Ibrahim Assem, Thomas Br{\"u}stle, Gabrielle Charbonneau-Jodoin, and Pierre-Guy
  Plamondon.
\newblock Gentle algebras arising from surface triangulations.
\newblock {\em Algebra Number Theory}, 4(2):201--229, 2010.

\bibitem{Bob1}
G.~Bobi{\'n}ski.
\newblock On the zero set of semi-invariants for regular modules over tame
  canonical algebras.
\newblock {\em J. Pure Appl. Algebra}, 212(6):1457--1471, 2008.

\bibitem{Bob-Sko-3}
G.~Bobi{\'n}ski and A.~Skowro{\'n}ski.
\newblock Geometry of directing modules over tame algebras.
\newblock {\em J. Algebra}, 215(2):603--643, 1999.

\bibitem{BS1}
G.~Bobi{\'n}ski and A.~Skowro{\'n}ski.
\newblock Geometry of modules over tame quasi-tilted algebras.
\newblock {\em Colloq. Math.}, 79(1):85--118, 1999.

\bibitem{Bob-Sko-2}
G.~Bobi{\'n}ski and A.~Skowro{\'n}ski.
\newblock Geometry of periodic modules over tame concealed and tubular
  algebras.
\newblock {\em Algebr. Represent. Theory}, 5(2):187--200, 2002.

\bibitem{Boh}
Ch. B{\"o}hning.
\newblock The rationality problem in invariant theory.
\newblock Preprint, ar\textrm{X}iv:0904.0899v1 [math.AG], 2009.

\bibitem{Bri}
M.~Brion.
\newblock Invariants et covariants des groupes alg{\'e}briques r{\'e}ductifs.
\newblock In {\em Th{\'e}orie des invariants et G{\'e}om{\'e}trie des
  vari{\'e}t{\'e}s quotients}, volume~61 of {\em Travaux en cours}. Herman,
  2010.

\bibitem{BR}
M.~C.~R. Butler and C.~M. Ringel.
\newblock Auslander-{R}eiten sequences with few middle terms and applications
  to string algebras.
\newblock {\em Comm. Algebra}, 15(1-2):145--179, 1987.

\bibitem{Carroll2}
A.~Carroll.
\newblock Generic modules for string algebras.
\newblock Preprint available at arXiv:1111.5064 [math.RT], 2011.

\bibitem{Carroll1}
A.~Carroll and J.~Weyman.
\newblock Semi-invariants for gentle string algebras.
\newblock Preprint available at arXiv:1106.0774 [math.RT], 2011.

\bibitem{Carroll:Thesis}
Andrew~T. Carroll.
\newblock {\em Semi-Invariants for Gentle String Algebras}.
\newblock PhD thesis, Northeastern University, 2012.

\bibitem{CC9}
C.~Chindris.
\newblock Geometric characterizations of the representation type of hereditary
  algebras and of canonical algebras.
\newblock {\em Adv. Math.}, 228(3):1405--1434, 2011.

\bibitem{CC10}
C.~Chindris.
\newblock On the invariant theory for tame tilted algebras.
\newblock To appear in Algebra and Number Theory. Preprint available at
  arXiv:1109.2915v1 [math.RT], 2012.

\bibitem{CC12}
C.~Chindris, R.~Kinser, and R.~Weyman.
\newblock Module varieties and representation type of finite-dimensional
  algebras.
\newblock Preprint available at arXiv:1201.6422v1 [math.RT], 2012.

\bibitem{C-BS}
W.~Crawley-Boevey and J.~Schr{\"o}er.
\newblock Irreducible components of varieties of modules.
\newblock {\em J. Reine Angew. Math.}, 553:201--220, 2002.

\bibitem{CB5}
W.~W. Crawley-Boevey.
\newblock On tame algebras and bocses.
\newblock {\em Proc. London Math. Soc. (3)}, 56(3):451--483, 1988.

\bibitem{DeConcini-Strickland}
Corrado De~Concini and Elisabetta Strickland.
\newblock On the variety of complexes.
\newblock {\em Adv. in Math.}, 41(1):57--77, 1981.

\bibitem{delaP}
J.~A. de~la Pe{\~n}a.
\newblock On the dimension of the module-varieties of tame and wild algebras.
\newblock {\em Comm. Algebra}, 19(6):1795--1807, 1991.

\bibitem{DW4}
H.~Derksen and J.~Weyman.
\newblock Semi-invariants for quivers with relations. {S}pecial issue in
  celebration of {C}laudio {P}rocesi's 60th birthday.
\newblock {\em J. Algebra}, 258(1):216--227, 2002.

\bibitem{Domo}
M.~Domokos.
\newblock Relative invariants for representations of finite dimensional
  algebras.
\newblock {\em Manuscripta Mat.}, 108:123--133, 2002.

\bibitem{Domo2}
M.~Domokos.
\newblock On singularities of quiver moduli.
\newblock Preprint available at arXiv:0903.4139v2 [math.RT], 2009.

\bibitem{DowSko2}
P.~Dowbor and A.~Skowro{\'n}ski.
\newblock On the representation type of locally bounded categories.
\newblock {\em Tsukuba J. Math.}, 10(1):63--72, 1986.

\bibitem{Dro}
Yu.~A. Drozd.
\newblock Tame and wild matrix problems.
\newblock In {\em Representations and quadratic forms (Russian)}, pages 39--74,
  154. Akad. Nauk Ukrain. SSR Inst. Mat., Kiev, 1979.

\bibitem{GeiSch}
Ch. Geiss and J.~Schr{\"o}er.
\newblock Varieties of modules over tubular algebras.
\newblock {\em Colloq. Math.}, 95(2):163--183, 2003.

\bibitem{K}
A.D. King.
\newblock Moduli of representations of finite-dimensional algebras.
\newblock {\em Quart. J. Math. Oxford Ser.(2)}, 45(180):515--530, 1994.

\bibitem{Kra}
W.~Kra\'skiewicz.
\newblock On semi-invariants of tilted algebras of type
  \textrm{$\mathbb{A}_n$}.
\newblock {\em Colloq. Math.}, 90(2):253--267, 2001.

\bibitem{Kraskiewicz:2011aa}
Witold Kra{\'s}kiewicz and Jerzy Weyman.
\newblock Generic decompositions and semi-invariants for string algebras.
\newblock Preprint available at arXiv:1103.5415, 2011.

\bibitem{PopVin}
V.~L. Popov and E.~B. Vinberg.
\newblock Invariant theory.
\newblock In I.~R. Shafarevich, editor, {\em Algebraic geometry. {IV}},
  volume~55 of {\em Encyclopaedia of Mathematical Sciences}, pages vi+284.
  Springer-Verlag, Berlin, 1994.
\newblock Linear algebraic groups. Invariant theory, A translation of {\it
  Algebraic geometry. 4} (Russian), Akad. Nauk SSSR Vsesoyuz. Inst. Nauchn. i
  Tekhn. Inform., Moscow, 1989 [ MR1100483 (91k:14001)], Translation edited by
  A. N. Parshin and I. R. Shafarevich.

\bibitem{Rei1}
Z.~Reichstein.
\newblock On the notion of essential dimension for algebraic groups.
\newblock {\em Transform. Groups}, 5(3):265--304, 2000.

\bibitem{Rie}
Ch. Riedtmann.
\newblock Tame quivers, semi-invariants, and complete intersections.
\newblock {\em J. Algebra}, 279(1):362--382, 2004.

\bibitem{Rie-Zwa-1}
Ch. Riedtmann and G.~Zwara.
\newblock On the zero set of semi-invariants for tame quivers.
\newblock {\em Comment. Math. Helv.}, 79(2):350--361, 2004.

\bibitem{Rie-Zwa-2}
Ch. Riedtmann and G.~Zwara.
\newblock The zero set of semi-invariants for extended {D}ynkin quivers.
\newblock {\em Trans. Amer. Math. Soc.}, 360(12):6251--6267, 2008.

\bibitem{R4}
C.~M. Ringel.
\newblock The rational invariants of the tame quivers.
\newblock {\em Invent. Math.}, 58(3):217--239, 1980.

\bibitem{Ros}
M.~Rosenlicht.
\newblock Some basic theorems on algebraic groups.
\newblock {\em Amer. J. Math.}, 78:401--443, 1956.

\bibitem{S4}
A.~Schofield.
\newblock Birational classification of moduli spaces of vector bundles over
  {$\Bbb P\sp 2$}.
\newblock {\em Indag. Math. (N.S.)}, 12(3):433--448, 2001.

\bibitem{Sim-Sko-3}
D.~Simson and A.~Skowro{\'n}ski.
\newblock {\em Elements of the representation theory of associative algebras.
  {V}ol. 3}, volume~72 of {\em London Mathematical Society Student Texts}.
\newblock Cambridge University Press, Cambridge, 2007.
\newblock Representation-infinite tilted algebras.

\bibitem{SW1}
A.~Skowro{\'n}ski and J.~Weyman.
\newblock The algebras of semi-invariants of quivers.
\newblock {\em Transform. Groups}, 5(4):361--402, 2000.

\end{thebibliography}
\end{document}